\documentclass[11pt,twoside]{preprint}
\usepackage[a4paper,innermargin=1.2in,outermargin=1.2in,
bottom=1.5in,marginparwidth=1in,marginparsep
=3mm]{geometry}
\usepackage{amsmath,amsthm,amssymb,enumerate}
\usepackage{hyperref}
\usepackage{mathrsfs}
\usepackage{breakurl}
\usepackage[nameinlink,capitalize]{cleveref}
\usepackage[title]{appendix}

\usepackage{newtxtext}

\usepackage{xcolor}
\usepackage{mhequ}
\usepackage{comment}
\usepackage{pifont}

\definecolor{Brown}{rgb}{.75,.5,.25}
\definecolor{DGreen}{rgb}{0,0.55,0}
\definecolor{Olive}{rgb}{0.41,0.55,0.13}

\newtheorem{theorem}{Theorem}[subsection]
\newtheorem{lemma}[theorem]{Lemma}
\newtheorem{proposition}[theorem]{Proposition}
\newtheorem{corollary}[theorem]{Corollary}

\theoremstyle{definition}

\newtheorem{definition}[theorem]{Definition}

\theoremstyle{remark}
\newtheorem{remark}[theorem]{Remark}

\def\scal#1{\langle #1 \rangle}

\newcommand{\vn}[1]{{\vert\kern-0.23ex\vert\kern-0.23ex\vert #1 
    \vert\kern-0.23ex\vert\kern-0.23ex\vert}}
\usepackage{mathtools}

\allowdisplaybreaks

\newcommand{\A}{\mathcal{A}}
\newcommand{\C}{\mathcal{C}}
\newcommand{\F}{\mathcal{F}}
\newcommand{\bF}{\mathbb{F}}

\newcommand{\bE}{\mathbb{E}}
\newcommand{\bP}{\mathbb{P}}

\newcommand{\cE}{\mathcal{E}}

\newcommand{\cS}{\mathcal{S}}
\newcommand{\cB}{\mathcal{B}}

\newcommand{\N}{\mathbb{N}}
\newcommand{\R}{\mathbb{R}}
\newcommand{\Z}{\mathbb{Z}}
\newcommand{\T}{\mathbb{T}}
\newcommand{\bB}{\mathbb{B}}

\def\cG{\mathcal{G}}
\newcommand{\eps}{\varepsilon}
\def\E{\hskip.15ex\mathbb{E}\hskip.10ex}

\renewcommand{\d}{\partial}

\newcommand{\id}{{\mathrm{id}}}

\newcommand{\Binom}{{\mathrm{Binom}}}

\newcommand{\D}{\partial}
\newcommand{\cP}{\mathcal{P}}

\newcommand{\cL}{\mathcal{L}}
\newcommand{\cF}{\mathcal{F}}
\newcommand{\cA}{\mathcal{A}}
\newcommand{\cC}{\mathcal{C}}
\newcommand{\scC}{\mathscr{C}}
\newcommand{\bone}{\mathbf{1}}

\def\({\left(}
\def\){\right)}

\begin{document}
\title{Optimal rate of convergence for approximations of SPDEs with non-regular drift}
\author{Oleg Butkovsky\thanks{Weierstrass Institute, Mohrenstra\ss e 39, 10117 Berlin, Germany$\qquad$\url{oleg.butkovskiy@gmail.com}}\,,
Konstantinos Dareiotis\thanks{University of Leeds, Woodhouse, LS2 9JT Leeds, United Kingdom$\qquad$\url{k.dareiotis@leeds.ac.uk}}, and M\'at\'e Gerencs\'er\thanks{TU Wien, Wiedner Hauptstra\ss e 8-10, 1040 Vienna, Austria$\qquad$\url{mate.gerencser@tuwien.ac.at}}}
\maketitle
\begin{abstract}
A fully discrete finite difference scheme for stochastic reaction-diffusion equations driven by a $1+1$-dimensional white noise is studied. The optimal strong rate of convergence is proved without posing any regularity assumption on the non-linear reaction term. The proof relies on stochastic sewing techniques.
\end{abstract}

\bigskip

\bigskip
\tableofcontents

\section{Introduction}
Consider the stochastic partial differential equation (SPDE)

\begin{equ}\label{eq:main in classical form}
\d_t u=\Delta u+b(u)+\xi\qquad \text{on } (0,\infty)\times \T,\qquad\qquad u_0=\psi\qquad\text{on } \T.
\end{equ}
Here the unknown $u$ is a random space-time stochastic process in $1+1$ dimensions, $\xi$ is a space-time white noise, and $b:\R\to\R$ is a given function.
The spatial domain is the $1$-dimensional torus $\T=\R/\Z$, in other words, we consider the equation with periodic boundary conditions.
Owing to the regularising property of the noise, equation \eqref{eq:main in classical form} is well-posed even with merely bounded and measurable $b$,
as classical results of Gy\"ongy and Pardoux \cite{GyP1,GyP2} show.
For a far-reaching generalisation of these results we refer to the recent work \cite{ABLM}.

The error analysis of stochastic reaction-diffusion equations of the form \eqref{eq:main in classical form} with various regularity assumptions on the drift $b$ goes back to the early days of numerical analysis of SPDEs.
In what was the first study of a fully discrete numerical scheme for SPDEs, Gy\"ongy \cite{Gy} showed\footnote{\cite{Gy} considers \eqref{eq:main in classical form} with Dirichlet boundary conditions instead of periodic.} that the space-time finite difference approximation of the above equation (A) strongly converges to the true solution if $b$ is a bounded measurable function (B) converges with strong rate $1/4$ w.r.t. time and $1/2$ w.r.t. space if $b$ is a Lipschitz continuous function.
This rate was in fact shown to be sharp by Davie and Gaines \cite{Sandy-SPDE},
who proved matching lower bounds even in the linear case $b\equiv 0$.
Despite a rapidly growing literature on the numerics of SPDEs in the two decades since, the ``gap'' between (A) and (B) has remained
and no rate of convergence has been known even if $b$ is just shy of Lipschitz: say, $b\in\cC^\alpha$ with $\alpha<1$.

The aim of this paper is to resolve this question and derive the  optimal rate of convergence (up to loss of arbitrarily small $\eps$)  without any regularity assumption on $b$. The main result can be informally summarized as follows.
For the precise statement we refer to Theorem \ref{thm:main-theorem}.
\begin{theorem}\label{thm:informal}
For any $\eps\in(0,1/2)$, bounded and measurable $b$, and any initial condition of class $\cC^{1/2-\eps}(\T)$, the forward Euler finite difference approximation of \eqref{eq:main in classical form} converges strongly with rate $1/4-\eps/2$ w.r.t. time and $1/2-\eps$ w.r.t. space.
\end{theorem}

The strategy of the proof is quite different from previous works.
In  \cite{Gy}, the method for  the bounded $b$ case
crucially relies on the Gy\"ongy-Krylov lemma \cite[Lem. 1.1]{GyK} and therefore is inherently not quantitative.
As for methods in the Lipschitz (or one-sided Lipschitz) $b$ case (see below for some references), they build on the analysis of the corresponding deterministic problem.
Such approach is out of question for $b \in \C^\alpha$, $\alpha<1$, since without the noise the PDE is not even well-posed, in general.
Instead, our strategy uses stochastic sewing, initiated in \cite{Khoa} and further 
developed in the numerical analytic direction in \cite{BDG,KMK}.

\subsection{Literature}
As mentioned above, quantitative results have so far remained out of reach when $b$ is even slightly irregular, i.e. not at least one-sided Lipschitz.
Qualitative results, further to the works of Gy\"ongy \cite{Gy-spatial,Gy},
were obtained in the case of bounded measurable $b$ in Pettersson and Signahl \cite{Pettersson2005} (convergence in the nondegenerate multiplicative case)
and in Anton, Cohen, and Quer-Sardanyons \cite{Cohen} (convergence for an exponential integrator scheme).
Needless to say, in the case of regular coefficients the rate of convergence of various discretisations of SPDEs is extensively studied.
Even just in the context of space-time white noise driven reaction-diffusion equations the literature is rich, see among others \cite{+B,+B2,+D,+J,+L,+P,+S,+W}.
A wider overview can be found for example in the above mentioned work \cite[Sec. 1]{Cohen} or in Da Prato-Zabczyk \cite[Sec. 14.1.10]{DPZ}. The interested reader is also referred to the monographs \cite{Jentzen-book, Kruse-book}.

In contrast to SPDEs,  which can be seen as infinite dimensional SDEs, the question of rate of convergence for finite dimensional SDEs with irregular drift coefficient is far more well-studied.
As a small sample, we mention some of the most recent works \cite{BDG, NeuSz, MY-lowerbound, Menozzi21, Taguchi, Y-order1, LL}.
The developments of the last years are discussed in more detail in the survey \cite{survey}.
However, we mention that even in the finite dimensional case, the optimal strong convergence rate without any regularity assumptions has only been proved quite recently \cite{KMK}.

\subsection{Notation}

For a metric space $(X,d)$  we define the following spaces of $\R$-valued functions. The space of bounded and Borel-measurable functions is denoted by $\bB(X)$ and is equipped with the norm $\|f\|_{\bB(X)}=\sup_x|f(x)|$.
The space of continuous functions is denoted by $\mathcal{C}(X)$.
For $\alpha\in(0,1]$ we denote by $\cC^\alpha(X)$ the space of bounded functions $f$ that satisfy
\begin{equ}
\,[f]_{\cC^\alpha(X)}:=\sup_{x\neq y}\frac{|f(x)-f(y)|}{d(x,y)^\alpha}<\infty.
\end{equ}
We equip $\cC^\alpha(X)$ with the norm $\|f\|_{\cC^\alpha(X)}=\|f\|_{\bB(X)}+[f]_{\cC^\alpha(X)}$.
By convention, we set $\cC^0(X):=\bB(X)$ (and not $\cC(X)!$).

We fix a probability space $(\Omega,\cF,\bP)$. The white noise $\xi$ on $[0,\infty)\times\T$ is a mapping from $\cB_b([0,\infty)\times \T)$, the bounded Borel sets of $[0,\infty )\times \T$, to $L_2(\Omega)$ such that for any collection $A_1,\ldots,A_k$ of elements of $\cB_b([0,\infty)\times\T)$, the vector $\big(\xi(A_1),\ldots,\xi(A_k)\big)$ is Gaussian with mean $0$ and covariance $\E\big(\xi(A_i)\xi(A_j)\big)=|A_i\cap A_j|$, where $|\cdot|$ denotes the Lebesgue measure. 
We  also fix a filtration $\bF=(\cF_t)_{t\in[0,1]}$ such that for each $t\geq 0$, $A\in\cB_b([0,t]\times\T)$, and $B\in\cB_b([t,\infty)\times\T)$, the random variable $\xi(A)$ is $\cF_t$-measurable, $\xi(B)$ is independent of $\cF_t$, and $\cF_t$ is $\bP$-complete.
For example, we may (but don't necessarily have to) take $\bF$ to be the completion of the filtration generated by $\xi$. The predictable $\sigma$-algebra on $\Omega \times [0,1]$ is denoted by $\mathscr{P}$. 
The conditional expectation given $\cF_t$ is denoted by $\E^t$.
The space-time stochastic integrals with respect to $\xi$ are denoted by
\begin{equ}
\int_0^t\int_\T f(s,y)\,\xi (dy, ds) =\int_0^1\int_\T\bone_{s\in[0,t]}f(s,y)\,\xi (dy, ds) .
\end{equ}
Most of the time the integrand $f$ will be deterministic, in which case the stochastic integral can simply be defined as the continuous and linear extension of the mapping $\bone_{A}\mapsto\xi(A)$ from $L_2([0,1]\times\T)$ to $L_2(\Omega)$, which is in fact an isometry.
More generally,   we might consider $\mathscr{P} \otimes \mathcal{B}(\T)$-measurable  integrands $f : \Omega\times  [0, 1] \times \T \to \R$, with $f \in  L_2(\Omega\times  [0, 1] \times \T)$.   Their stochastic integration can be found in e.g. \cite{DPZ}.

We denote the convolution operator by
\begin{equ}
(f\ast g)(x)=\int f(x-y)g(y)\,dy.
\end{equ}
This notation is used both when the domain of integration is $\R$ and $\T$. Since in a typical situation $f$ will be a heat kernel either on $\R$ or $\T$, the context will make it clear which convolution we mean.

In proofs of theorems/lemmas/propositions we
use the shorthand $f\lesssim g$ to mean that there exists a constant $N$ such that $f\leq N g$, and that $N$ does
not depend on any other parameters than the ones specified in the theorem/lemma/proposition.  Moreover, when we explicitly write $f \leq Ng$,  again the constant $N$ depends only on the parameters stated in the corresponding theorem/lemma/proposition and might change from line to line. 

\subsection{Formulation}

We consider the finite difference, forward Euler approximation of \eqref{eq:main in classical form}.
To this end, we introduce the space and time grids, for each $n\in\N=\{1,2,3,\ldots\}$
\begin{equs}
\Pi_n&=\{0,(2n)^{-1},\ldots,(2n-1)(2n)^{-1}\},\qquad
\Lambda_n= \{0, c (2n)^{-2}, 2c(2n)^{-2},\ldots\},
\end{equs}
where $c$ is a constant satisfying the condition $c\in(0,1/2)$, also  commonly known as the Courant–Friedrichs–Lewy (CFL) condition in the present context.
\begin{remark}
The restriction to look at spatial grids with even number of points (i.e. the choice of $2n$) is purely for convenience, otherwise the even and odd cases would require some notational distinction later on. The choice of focusing on the even case is motivated by the computational practice of using nested grids of mesh sizes $2^{-k}$, $k=1,2,\ldots, N$, up to some threshold $N$.
\end{remark}
Note also that on $\Pi_n$, just like on $\T$,  the addition is understood in a periodic way, i.e. $(2n-1) (2n)^{-1}+(2n)^{-1}$ is identified with $0$.
To ease notation, we also denote $h=c(2n)^{-2}$. Hence, by setting $\N_0=\N\cup\{0\}$, one has $\Lambda_n=h\N_0$.
Take an approximate initial condition $\psi^n : \Omega \times  \T \to \R$.
The approximation scheme is defined by setting
$u^n_0( x)=\psi^n( x)$ for $ x\in\Pi_n$ and
then inductively
\begin{equs}
u^n_{{t}+h}( {x}) &=u^n_{{t}}( {x})+h \Delta_n u^n_{ t}( x)
+h b(u^n_{ t}({x}))+h  \eta_n ({t}, {x}) \label{eq:approx classical form}
\end{equs}
for $ t\in\Lambda_n$ and $ x\in\Pi_n$, where
the discrete Laplacian is defined as
\begin{equ}
\Delta_n f( x)=(2n)^2\big(f( x+(2n)^{-1})-2f( x)+f( x-(2n)^{-1})\big)\,,
\end{equ}
and the discrete noise term is given by
\begin{equ}
\eta_n ({t}, {x})=2n h^{-1}\xi\Big([ t, t+h]\times[ x, x+(2n)^{-1}]\Big).
\end{equ}
Recall that \eqref{eq:main in classical form} admits a unique mild solution (see  Definition \ref{def:mild} below) which will be denoted by $u$.  The main result of the article reads as follows.
\begin{theorem}  \label{thm:main-theorem}
Let $p\geq 2$, $\eps \in (0, 1/4)$, and let $b$ be bounded and measurable. Assume that  the initial conditions $\psi,\psi^n$ are $\cF_0$-measurable $\cC^{1/2-\eps}(\T)$-valued random variables,  such that for a constant $K<\infty$ they  satisfy $\|\psi\|_{L_p(\Omega; \cC^{1/2-\eps}(\T))},\|\psi^n\|_{L_p(\Omega; \cC^{1/2-\eps}(\Pi_n))}\leq K$.
Then there exists a constant $N$ depending only on the parameters $c,p,\eps,K,\|b\|_{\bB(\R)}$ such that for all $n\in\N$ the following bound holds:
\begin{equ}\label{eq:main-estimate}
\sup_{(t,x)\in([0,1]\cap\Lambda_n)\times\Pi_n}\|u_t(x)-u^n_t(x)\|_{L_p(\Omega)}\leq N \big(n^{-1/2+\eps}+\sup_{x\in\T}\|\psi(x)-\psi^n(x)\|_{L_p(\Omega)}\big).
\end{equ}
\end{theorem}

\begin{remark}
As will follow from the proof, with an appropriate extension of $u^n$ from the gridpoints $\Lambda_n\times\Pi_n$ to the whole of $\R_+\times\T$, the supremum on the left-hand side of \eqref{eq:main-estimate} can be taken over $(t,x)\in[0,1]\times\T$.
\end{remark}
\begin{remark}
The freedom of allowing different initial condition for the approximation is not a particularly important feature of the statement, but it is convenient for the proof. Indeed, it allows to easily deduce the general case from the case of short times, that is, when the supremum runs only over $t\leq S$, where $S$ is a (small) constant depending only on the parameters of the problem.
\end{remark}

\begin{remark}
There are several natural interesting directions  to generalise the results of the present article: for example, equations driven  by spatially coloured/multiplicative/L\'evy noises,  distributional drifts or even different approximation schemes. We leave these for future work.
\end{remark}

\bigskip

\noindent\textbf{Acknowledgments.} 
The authors thank the referees for numerous constructive suggestions.
MG was funded in part by the Austrian Science Fund (FWF) Stand-Alone programme P 34992. For the purpose of open access, the authors have applied a CC BY public copyright license to any Author Accepted Manuscript version arising from this submission.
We further thank the financial support of the International Office at TU Wien.

\section{Estimates on heat kernels and stochastic convolutions}
The evolution of the true and of the approximate solutions is very different even in the linear case $b=0$.
This is one of the main challenges compared to the finite dimensional case, where with vanishing drift the two processes are simply given by the noise process and in particular the error is $0$ (see Section \ref{sec:outline} for a bit more detailed comparison to the finite dimensional case). In infinite dimensions, 
the error of the linear problem propagates in a nontrivial way in the error analysis of the case of irregular $b$.
The aim of this section is therefore to derive various estimates for the continuous and discrete heat kernels and the associated Ornstein-Uhlenbeck processes (i.e. the solutions of \eqref{eq:main in classical form}, \eqref{eq:approx classical form} in the case $\psi=0$, $b=0$).
\subsection{Definitions}
We encounter three different heat kernels in the article: the continuum heat kernel on $\R$, the continuum heat kernel on $\T$, and the discrete heat kernel on $\T$.
The first two are defined by
\begin{equs}
p_t^\R(x)&=\frac{1}{\sqrt{2\pi}t}\exp\Big(-\frac{x^2}{2t}\Big),
\\
p_t(x)&=\sum_{k\in\Z}\frac{1}{\sqrt{4\pi}t}\exp\Big(-\frac{(x+k)^2}{4t}\Big)= \sum_{k \in \Z} e^{- 4 \pi^2 k^2 t } e^{i2\pi k x},
\end{equs}
the last equality obtained by Poisson summation.
The difference in scaling comes from the fact that the first is chosen to be the density function in $x$ of a centered normal random variable with variance $t$, while the latter is chosen to be the Green's function of the heat operator on $\R\times\T$.
For sake of convenience we use separate notation for the action of the heat kernels via convolution: for $f\in\bB(\T)$, we denote $\cP_tf:=p_t\ast f$. We define $\cP_t^\R$ analogously.
The continuum heat kernels form a semigroup, that is, $\cP_t(\cP_s f)=\cP_{t+s}f$, and similarly for $\cP^\R$.
The periodic heat kernel is used for the definition of the mild solution of \eqref{eq:main in classical form}.
\begin{definition}        \label{def:mild}
A mild solution of \eqref{eq:main in classical form} is  a  $\mathscr{P} \otimes \mathcal{B} (\T)$-measurable map $u : \Omega \times [0,1] \times \T \to \R$  which is  continuous in $(t,x)$, 
such that almost surely for all $(t,x)\in[0,1]\times \T$ the following equality holds:
\begin{equ}\label{eq:main in mild form}
u_t(x)=\cP_t \psi(x)+\int_0^t\cP_{t-s} b(u_s)(x)\,ds+\int_0^t\int_\T p_{t-s}(x-y)\,\xi (dy, ds) .
\end{equ}
\end{definition}

The setup of the discrete heat kernels is more complicated.
The formulation below follows along the lines of \cite{Gy}, but for the convenience of the reader and due to various small differences we prefer to give the full details.
From now on,  the conjugate of a complex number $z\in\mathbb{C}$ is denoted by $\overline{z}$. Consider the functions $e_j(x)= e^{i 2 \pi j x}$ for $j\in\Z$.  
They are eigenfunctions of $\Delta$ with eigenvalues $\lambda_j= -4 \pi^2j^2$.  It is  well-known that $(e_j)_{j \in \Z}$ forms an orthonormal basis of $L^2(\T; \mathbb{C})$.
In the next proposition we prove a discrete analogue.  It will be convenient to use the piecewise linear extension of the restriction of $e_j$ to $\Pi_n$: for $-n\leq j\leq n-1$, for $x\in \Pi_n$, and $x' \in [x, x+(2n)^{-1}]$, set
\begin{equs}      \label{eq:linear-interpolation}
e_j^n(x')= e_j(x)+2n(x'-x)(e_j(x+(2n)^{-1})-e_j(x)).
\end{equs}
\begin{proposition}
Let $\lambda^n_j = -16 n^2 \sin^2\big( \frac{j \pi}{2n}\big)$ for $j\in \Z$. Then
\begin{equs}    \label{eq:discrete-eigenstuff}
\Delta_n e_j(x) = \lambda^n_j e_j(x)
\end{equs}
for $x \in \Pi_n$. 
Moreover, if $-n\leq j,\ell\leq n-1$, then
\begin{equs}
\frac{1}{2n}\sum_{ x\in\Pi_n} e_j( x)  \overline{e_\ell( x)}  = \bone_{j=\ell}.\label{eq:orthogonal-n}
\end{equs}
As a consequence, $e_{-n}^n,e_{-n+1}^n,\ldots,e_{n-1}^n$, as functions on $\Pi_n$ form a basis of $L^2(\Pi_n; \mathbb{C})$.
\end{proposition}
\begin{remark}
Of course $L^2(\Pi_n; \mathbb{C})$ can be simply identified with $\mathbb{C}^{2n}$, but keeping the former viewpoint is more instructive.
\end{remark}

\begin{proof}
We start with \eqref{eq:discrete-eigenstuff}. For $j \in \Z$   and $x = k/2n$, we have 
\begin{equs}
\Delta_n e_j(x)  =&  4n^2  \Big( \exp  \Big( \frac{i 2\pi j k  }{2n} + \frac{i 2 \pi j }{2n}\Big) -2 \exp  \Big( \frac{ i 2 \pi j k }{2n} \Big) + \exp  \Big( \frac{ i 2  \pi j k  }{2n}- \frac{i 2\pi j }{2n} \Big) \Big)
\\
= &  4 n^2 e_j(x) \Big( \exp  \Big(  \frac{i 2 \pi j }{2n}\Big) -2  + \exp  \Big( - \frac{i 2\pi j }{2n} \Big) \Big)
\\
=& 8  n^2 e_j(x) \Big( \cos \Big( \frac{ 2\pi j}{2n} \Big) -1\Big) = - 16   n^2 e_j(x)  \sin^2 \Big( \frac{ \pi j}{2n} \Big) = \lambda^n_j e_j(x) . 
\end{equs}
This proves \eqref{eq:discrete-eigenstuff}.
As for \eqref{eq:orthogonal-n}, the case $j=\ell$ is trivial. For $j\neq \ell$, using that $|j-\ell|<2n$, in the geometric series below the ratio $e^{\frac{i2\pi(j-\ell)}{2n}}$ is different from $1$, therefore
\begin{equ}
\sum_{ x\in\Pi_n} e_j( x)  \overline{e_\ell( x)}  = \sum_{k=0}^{2n-1} e^{\frac{i 2 \pi (j-\ell)k}{2n}}= \frac{1-e^{i 2\pi (j-\ell)}}{1- e^{\frac{i2\pi (j-\ell)}{2n}}}=0 .
\end{equ}
\end{proof}
Although \eqref{eq:discrete-eigenstuff} holds for all $j\in\Z$, in the sequel we will only ever consider $-n\leq j\leq n-1$.
Let us also briefly discuss how $\lambda_j^n$ relates to $\lambda_j$. Defining $\gamma_0^n=1$ and $\gamma_j^n=\frac{\lambda_j^n}{\lambda_j}=\frac{\sin^2(j\pi/2n)}{(j\pi/2n)^2}$, one has 
\begin{equs}    \label{eq:bound-gama^n_j}
4 \pi^{-2} \leq \gamma^n_j\leq 1.
\end{equs}
Indeed, it is elementary to see that $\frac{\sin^2(x)}{x^2}$ is even, decreasing on $[0,\pi]$, so its minimum on $[-\pi/2,\pi/2]$ equals $4\pi^{-2}$.
As a consequence,
\begin{equ}\label{eq:sunday}
\lambda^n_j\leq -16 j^2.
\end{equ}
Moreover, one has for $-n\leq j \leq n-1$, 
\begin{equs}         \label{eq:bound1-gamma}
|1- \gamma^n_j| \leq \frac{1}{3}\Big( \frac{j \pi}{2n } \Big)^2,
\end{equs}
which follows from the inequality $1-(\sin (x)/x)^2\leq (1/3)x^2$ (whose proof we leave as an exercise to the interested reader).

\begin{remark}
Although we do not discuss purely spatial discretisations, it is worth remarking that the above setup would already be enough to define the heat kernel for the spatially discretised operator $\d_t-\Delta_n$ in its spectral representation, which would take the form
\begin{equ}\label{eq:spatial HK}
\tilde p^n_t(x,y)=\sum_{j=-n}^{n-1}e^{-t\lambda_j^n}e_j^n(x)\overline{e_j^n(y)}
\end{equ}
for $t\geq 0$, $x,y,\in\Pi_n$.
\end{remark}

It remains to encode the temporal discretisation in the discrete heat kernel.
Naturally, on the temporal gridpoints $t=kh$ the factor $e^{t\lambda_j^n}$ in $\eqref{eq:spatial HK}$ is simply replaced by $(1+h\lambda_j^n)^k$. Between the gridpoints, we again interpolate linearly.
More precisely, for $j =-n,\ldots,n-1$, for $t\in \Lambda_n$, and $t' \in [t, t+h]$, set 
\begin{equs}       \label{eq:def-of-mu}
\mu^n_j(t')= (1+h \lambda^n_j)^{ t h^{-1}} + h^{-1} (t'-t) \big(   (1+h \lambda^n_j)^{ (t+h) h^{-1}} -  (1+h \lambda^n_j)^{ t h^{-1}} \big). 
\end{equs}
The following, which is based on the CFL condition, will be used frequently. 
\begin{lemma}
There exists $\delta_0>0,  \delta>0$,  depending only on $c$,  such that for all $n \in \mathbb{N},  j = -n, ..., n-1$, $t \geq 0$, we have 
\begin{equ}      \label{eq:Pn-exponential}      
|1+h \lambda^n_j|^{t  h^{-1}}\leq e^{\delta_0 t\lambda^n_j}
\leq e^{-t\delta j^2}.
\end{equ}
\end{lemma} 
\begin{proof}
Recall that $4c \in (0, 2)$.  First, we claim that  there exists a constant $\delta_0>0$ such that for $x\in[0,4c]$ one has $|1-x|\leq e^{-\delta_0 x}$.  Indeed,  if  $x \in [0, 1]$ we have $|1-x|= 1-x \leq e^{-x}$.  In case $4c >1$,  we have that $4c-1 =  e^{-\delta_0' 4/c}$   for $\delta_0' :=  -  (c/4 ) \log(4c-1)$,  and $|x-1|$ is increasing on $[1, 4c]$ while $ e^{-\delta_0' x}$ is decreasing, so that $|1-x| \leq e^{-\delta_0' x}$ on $[1, 4c]$.  Hence, the claim holds true with $\delta_0:= 1 \wedge \delta_0'$.  The first inequality  in \eqref{eq:Pn-exponential}, follows from the fact that 
 for all $n, j$ we have 
\begin{equs}
- h \lambda^n_j = \frac{c}{(2n)^2 } 16n^2 \sin^2 \big( \frac{j \pi}{2n} \big)  \in [0, 4c]. 
\end{equs}
The second inequality in \eqref{eq:Pn-exponential} holds with the choice $\delta := 16 \delta_0$, since $ \lambda^n_j \leq - 16  j^2$, see \eqref{eq:sunday}.
\end{proof}

We can now define the discrete heat kernel and rewrite the approximation scheme \eqref{eq:approx classical form} in a mild form.
Denote by $\kappa_n(t)=\lfloor th^{-1}\rfloor h $ and $\rho_n(x)=\lfloor x 2n  \rfloor (2n)^{-1}$ the leftmost gridpoint from $t$ in $\Lambda_n$ and from $x$ in $\Pi_n$, respectively.
We then set 
\begin{equs}\label{eq:Pn-def}
p^{n}_t(x,y) = \sum_{j=-n}^{n-1}  \mu^n_j (t) e^n_j(x) \overline{e^n_j(\rho_n(y))},
\end{equs}
which is now a function of $t\geq0$, $x,y\in\T$. 

\begin{remark}   
Although each $e_j^n$ is a $\mathbb{C}$-valued function, $p_t^n$ itself is $\R$-valued for all $t\geq 0$.
Indeed, first consider $x \in \Pi_n$, $y \in \T$. Since $\lambda^n_j=\lambda^n_{-j}$ and therefore $\mu_j^n(t)=\mu_{-j}^n(t)$, one sees
\begin{equ}
\sum_{j=-n-1}^{n-1} \mu^n_j(t) e^n_j(x) \overline{e^n_j(\rho_n(y))}  \in \R. 
\end{equ}
In addition, the restriction of $e_{-n}^n$ to $\Pi_n$ takes only $\pm 1$ values, so
$e^n_{-n}(x) \overline{e^n_{-n}(\rho_n(y))}  \in \R$, which combined with the above shows that 
$p_t^{ n}(x,y) \in \R$ for $x \in \Pi_n$, $y \in \R$. The same is true for all $x, y \in \T$, since $p_t^n(\cdot,y)$ is given by linear interpolation between its values on $\Pi_n$.
\end{remark}
Let us  introduce  the discrete convolution 
\begin{equ}\label{eq:convolve-n-def}
f\ast_n g(x):=\int_\T f(x,y)g\big(\rho_n(y)\big)\,dy. 
\end{equ}
Analogously to $\cP$, we then define the linear operators $\cP^n$ by setting $\cP^n_tf:=p^n_t\ast_n f$.
Most of the time we understand $\cP^n_t$ as an operator on $\bB(\T)$, but it can be seen as an operator on $\bB(\Pi_n)$ as well.
In the latter case, the identity 
\begin{equs}      \label{eq:P^n-rep}
\cP^n_h=\id+h\Delta_n
\end{equs}
holds\footnote{Note however that $\cP^n_0$ as an operator on $\bB(\T)$ does not equal the identity.}.
The inductive step \eqref{eq:approx classical form} of the finite difference scheme can therefore be written as 
\begin{equ}\label{eq:inductive mild}
u^n_{ t+h}( x)=\cP^{n}_{h} u^n_{ t}( x)+
h b(u^n_{ t}({x}))+h  \eta_n ({t}, {x}).
\end{equ}
To conclude to a form similar to \eqref{eq:main in mild form}, it remains to show the following simple property.
\begin{proposition}
For  $s,t \in \Lambda_n $,  $x, y \in \T$,  we have 
\begin{equ}\label{eq:kindof semigroup}
\big(p^{n}_{t} \ast_n p^{n}_{s}(\cdot,y)\big)(x)=p_{t+s}^{n}(x,y).
\end{equ}
\end{proposition}
\begin{proof}
By the definitions \eqref{eq:Pn-def}-\eqref{eq:convolve-n-def} and the orthogonality relation \eqref{eq:orthogonal-n} we can write
\begin{equs}
\big(&p^{n}_{t}\ast_n p^{n}_s(\cdot,y)\big)(x)
=\int_\T p_{t}^{n}(x,z) p_{s}^{n}(\rho_n(z),y)\,dz
\\
&=\sum_{j=-n}^{n-1}\sum_{k=-n}^{n-1}(1+h \lambda^n_j)^{{t}h^{-1}}(1+h \lambda^n_k)^{{s} h^{-1}}e^n_j(x) \overline{ e^n_k(\rho_n(y))} \int_\T e^n_k(\rho_n(z)) \overline{ e^n_j(\rho_n(z))} \,dz
\\
&=\sum_{j=-n}^{n-1}(1+h \lambda^n_j)^{({s}+{t}) h^{-1}}e^n_j(x) \overline{ e^n_j(\rho_n(y)) } ,
\end{equs}
as required.
\end{proof}
It follows that for $t \in \Lambda_n, x \in \Pi_n$,  \eqref{eq:approx classical form} can be equivalently written as
\begin{equ}
u^n_{ t}( x)=\cP^{n}_{ t} \psi^n( x)+\int_0^{ t}\cP^{n}_{\kappa_n( t-s)} b(u^n_{\kappa_n(s)})( x)\,ds
+\int_0^{ t}\int_\T p_{\kappa_n( t-s) }^{n}( x,y)\,\xi (dy, ds) ,
\end{equ}
Indeed, this clearly holds for $ t=0$ and for $0< t\in\Lambda_n$ it follows inductively from \eqref{eq:inductive mild} and \eqref{eq:kindof semigroup}.
Recalling that $p^{n}$ is defined for any space-time point, not just the ones on the grid, we then \emph{define} an extension of $u^n$ to the whole of $[0,1]\times\T$ by setting
\begin{equ}\label{eq:approx mild form}
u^n_{t}( x)=\cP^{n}_{t}\psi^n( x)+\int_0^{ t}\cP^{n}_{ \kappa_n(t-s)} b(u^n_{\kappa_n(s)})( x)\,ds
+\int_0^{t}\int_\T p_{ \kappa_n(t-s) }^{n}( x,y)\,\xi (dy, ds) .
\end{equ} 
\begin{remark}\label{rem:random-walk}
It is useful to note an alternative representation of $\cP^{n}$ as the transition kernel of a random walk  indexed by $\Lambda_n$.  Let $X_1,X_2,\ldots$ be i.i.d. random variables with distribution
\begin{equ}
\bP(X_1=0)=1-2c,\quad
\bP(X_1=1)=P(X_1=-1)=c.
\end{equ}
One can observe that the condition $c\leq 1/2$ is necessary in order for the above to be a probability distribution, while our stronger condition $c<1/2$ guarantees that the random walk is ``lazy''.
We then define $S_n=\sum_{i=1}^nX_i $ and for $ t\in \Lambda_n$, set $\widehat{S}^n_t=(2n)^{-1}S_{h^{-1}t}$. 
Then $\cP^n$ is the transition semigroup of $\widehat{S}^n$: for any  function $f:\Pi_n\to\R$ and any $x\in\Pi_n$,
\begin{equ}\label{eq:disc-semigroup}
\cP_t^n f(x)=\E \tilde f(x+\widehat{S}^n_ { t}),
\end{equ}
where $\tilde f$ is the $1$-periodic extension of $f$ from $\Pi_n$ to $(2n)^{-1}\Z$.
\end{remark}
\subsection{Discrete and continuous heat kernel bounds}
We start with three classical heat kernel bounds. 
\begin{lemma}
Let  $(\cS, \mathbb{D}) \in \{ (\cP, \T), (\cP^{\R}, \R)\}$. The  following hold.
\begin{enumerate}[(i)] 
\item  For all  $ \alpha , \beta \in [0, 1]$ with $\alpha \leq \beta$, there exists a constant $N=N(\alpha,\beta)$, such that for all $f \in \cC^\alpha(\mathbb{D})$, $0 
\leq  s \leq t \leq 1$ and $x, y  \in \mathbb{D} $  one has \footnote{with the conventions $0^0=1$, $1/0 = +\infty$,  $\mathcal{S}_0=\id$}
\begin{equ}\label{eq:HK bound mixed}
|\cS_t f(x)-\cS_s f(y)| \leq N\|f\|_{\C^\alpha(\mathbb{D})}\big(|x-y|^\beta +|t-s|^{\beta /2}\big)s^{(\alpha-\beta )/2}.
\end{equ}

\item For any  $\alpha\in[0,1]$,  there exists a constant $N=N(\alpha)$ such that for all $f\in\cC^\alpha(\mathbb{D})$, $ t\in(0,1]$,   and $x_1,x_2,x_3,x_4\in \mathbb{D}$ one has
\begin{equs}\label{eq:HK bound 3}
|\cS_t  & f(x_1)  -\cS_t f(x_2)-\cS_t f(x_3)+\cS_t f(x_4)|
\\
&
\leq N \|f\|_{\C^\alpha(\mathbb{D})}|x_1-x_2| |x_1-x_3| t^{\alpha/2-1}
+N\|f\|_{\C^\alpha(\mathbb{D})}|x_1-x_2-x_3+x_4| t^{(\alpha-1)/2}.
\end{equs}
\item There exists a constant $N$ such that for all $f\in\bB(\mathbb{D})$, $0<s\leq t\leq 1$, and $x, y \in \mathbb{D}$, one has
\begin{equs}        \label{eq:HK bound 4}
| \cS_t  f (x)-\cS_t  f (y)- \cS_s  f (x)- \cS_s  f (y)| \leq N \| f\|_{\mathbb{B}(\mathbb{D})} s^{-3/2} |t-s| |x-y|.
\end{equs}
\end{enumerate}
\end{lemma}

\begin{proof}
The estimate in \eqref{eq:HK bound mixed} is very standard.  A proof of it and its more general variants can be found for example in \cite[Appendix A]{BDG}.

For \eqref{eq:HK bound 3}, notice that  by the fundamental theorem of calculus we have 
\begin{equs}
 \cS_t &f(x_1)- \cS_t f(x_2)-\cS_t f(x_3)+\cS_t f(x_4)
\\
&=  \int_0^1\Big(  \nabla \cS_t f \big(x_1+\theta(x_2-x_1) \big) - \nabla \cS_t f \big(x_3+\theta(x_2-x_1) \big)  \Big) (x_1-x_2) \, d \theta 
\\
&\quad+ \cS_t f(x_3+x_2-x_1) - \cS_tf(x_4).
\end{equs}
Consequently, we get 
\begin{equs}
 |\cS_t & f(x_1)- \cS_t f(x_2)-\cS_t f(x_3)+\cS_t f(x_4)|
\\
&\leq   \|  \nabla ^2 \cS_t f \|_{\mathbb{B}(\mathbb{D})} |x_1-x_3| |x_1- x_2| + \|  \nabla \cS_t f \|_{\mathbb{B}(\mathbb{D})}|x_1-x_2-x_3+x_4|     \label{eq:idk1}
\end{equs}
From \eqref{eq:HK bound mixed}, with $\beta=1$, it follows that 
\begin{equs}                   
 \|  \nabla \cS_t f \|_{\mathbb{B}(\mathbb{D})} \lesssim t^{(\alpha-1)/2} \|f\|_{\C^\alpha(\mathbb{D})}.       \label{eq:idk2}
\end{equs}
Applying  \eqref{eq:idk2} twice (the first with the choice $\alpha=0$), we get
\begin{equs}                      
 \|  \nabla^2  \cS_t f \|_{\mathbb{B}(\mathbb{D})} =  \|  \nabla  \cS_{t/2}   \nabla( \cS_{t/2} f)  \|_{\mathbb{B}(\mathbb{D})}&  \lesssim t^{-1/2} \|  \nabla( \cS_{t/2} f)   \|_{\mathbb{B}(\mathbb{D})} 
 \\
 &  \lesssim t^{-1+\alpha/2} \|f\|_{\C^\alpha(\mathbb{D})}.     \label{eq:idk3}
\end{equs}
Therefore, from \eqref{eq:idk1}-\eqref{eq:idk3} we get \eqref{eq:HK bound 3}.

Finally, from the fundamental theorem of calculus, the identity $\D_t \cS= \Delta \cS$, 
\eqref{eq:HK bound mixed} with $\beta=1$, and \eqref{eq:idk3} with $\alpha=0$,  we  get 
\begin{equs}
 | \cS_t & f (x)-\cS_t  f (y)- \cS_s  f (x)- \cS_s  f (y)| 
\\
& =  \Big| (t-s) \int_0^1 \Big(  \Delta \cS_{ s+\theta(t-s)}  f (x)-  \Delta \cS_{ s+\theta(t-s)}  f (y) \Big) \, d \theta  \Big|
\\
& =  \Big| (t-s) \int_0^1 \Big(   \cS_{s/2+ \theta(t-s)} \Delta \cS_{s/2}  f (x)-  \cS_{s/2+ \theta(t-s)}  \Delta \cS_{s/2}  f (y) \Big) \, d \theta  \Big| 
\\
&  \lesssim  |x-y| |t-s|  s^{-1/2} \| \Delta S_{s/2} f \|_{\mathbb{B}(\mathbb{D})} 
\\
&  \lesssim  |x-y| |t-s| s^{-3/2} \| f\|_{\mathbb{B}(\mathbb{D})}. 
\end{equs}
This finishes the proof. 
\end{proof}
Before moving on, we remark a simple bound that will be frequently used.
\begin{proposition}\label{prop:summation}
For any $\lambda>0$ and $\gamma\geq 0$ there exists a $N=N(\lambda,\gamma)$ such that for all $t\in (0,1]$ the following bound holds
\begin{equ}\label{eq:summation}
\sum_{j\in\Z} |j|^{\gamma}e^{-\lambda j^2 t}\leq N t^{-(\gamma+1)/2}.
\end{equ}
\end{proposition}
\begin{proof}
By the change of variables $x \mapsto t^{-1/2}x$, we have 
\begin{equ}
\sum_{j\in\Z} |j|^{\gamma}e^{-\lambda j^2 t} \leq \int_{-\infty}^ \infty \big( |x|+1 \big)^\gamma e^{-\lambda x^2 t} \, dx =  t^{-(\gamma+1) /2} \int_{-\infty}^ \infty \big(  |x|+1 \big)^\gamma e^{-\lambda x^2 } \, dx.
\end{equ}
\end{proof}
As a toy example, one can see that with some absolute constant $N>0$ one has for all $s\in(0,1]$
\begin{equ}\label{eq:HK-L2}
N^{-1}s^{-1/2}\leq \|p_s\|_{L_2(\T)}^2=\sum_{k\in\Z}e^{-4\pi k^2s}\leq N^{-1}s^{-1/2},
\end{equ}
applying Proposition \ref{prop:summation} with $\gamma=0$ for the second inequality.
The first inequality in \eqref{eq:HK-L2} follows from bounding the sum from below by its restriction to $|k|\leq s^{-1/2}$, so that each of the $1+\lfloor s^{-1/2}\rfloor$ terms are bounded from below by $e^{-4\pi}$.

\begin{remark}\label{rem:interpolation}
Another useful tool that will be repeatedly used is interpolation.
It is known that there exists an absolute constant $N$ such that for all $n\in\N$ and $\alpha\in(0,1)$ one has
\begin{equ}
\sup_{\lambda\in(0,1]}\lambda^{-\alpha}\inf_{ \substack{ f_0, f_1 : \Pi_n \to \R\\ f=f_0+f_1}}  ( \| f_0 \|_{\bB(\Pi_n)} + \lambda \| f_1 \|_{\C^1(\Pi_n)} )\leq N\|f\|_{\C^\alpha(\Pi_n)},
\end{equ}
see \cite[Example 1.1.8]{Lunardi} for a proof that carries through in our setting (equivalently, simply bound the infimum by the choice when $f_1$ is the convolution of $f$ with the indicator of $[-\lambda/4,\lambda/4]\cap\Pi_n$).
It is a well-known consequence of the above bound (see e.g. \cite[Theorem 1.1.6]{Lunardi}) that if a linear operator $T$ mapping from $C^1(\Pi_n)$ to a normed vector space $V$ satisfies $\|T f\|_V\leq K_0 \|f\|_{\bB(\Pi_n)}$ and $\|T f\|_V\leq K_1 \|f\|_{\C^1(\Pi_n)}$ with some constants $K_0,K_1$ for all $f\in\bB(\Pi_n)$, then for any $\alpha\in(0,1)$ the bound $\|T f\|_V\leq K_0^{1-\alpha}K_1^\alpha \|f\|_{\C^\alpha(\Pi_n)}$ also holds.
\end{remark}

Heat kernel estimates for the discrete heat kernels $\cP^n$ are less established.
Since they are piecewise linear in time on each interval between neighbouring gridpoints, most estimates will be stated only for $t\in\Lambda_n$.
For the initial time we only need the straightforward property
\begin{equ}\label{eq:Pn-initial-L2}
\| p^n_0(x,\cdot)\|^2_{L_2(\T)}=2n.
\end{equ}
For gridpoints after the initial time we recover almost the usual heat kernel bounds, at the cost of a $\log$ factor.
\begin{lemma}    \label{lem:Discrete-HK-bounds}
Let $\alpha, \beta \in [0,1]$ with  $ \alpha \leq \beta$. There exists a constant $N=N(\alpha, \beta, c)$ such that for all $f:\Pi_n\to\R$, $n\in\N$,  $x, z \in \Pi_n $, $t \in \Lambda_n\cap(0,1]$, the following bound holds 
\begin{equs}     \label{eq:discrete-HK-bounds}
| \cP_{t}^{n} f (x) -\cP_{t}^{n} f (z) | \leq  N (\log(2n))^{(\beta-\alpha)/2 }|t|^{(\alpha-\beta)/2} | x- z |^\beta  \|f\|_{\cC^\alpha (\Pi_n)}.
\end{equs}
\end{lemma}

\begin{proof}
First assume $ \alpha=0 $ and $\beta=1$.
Note that in this case it suffices to consider neighbouring points $x,z\in\Pi_n$ (by virtue of the triangle inequality) and in fact only the case $x=0$ and $z=1/2n$ (by virtue of translation invariance of the H\"older norm).
 For a function $g : \Pi_n \to \R$, denote $\|g\|_{L_1(\Pi_n)}=(2n)^{-1}\sum_{x\in\Pi_n}|g(x)|$.
We then have 
\begin{equs}
 | \cP_{t}^{n} g (0) -\cP_{t}^{n} g (1/2n) |
&= 
\Big| \sum_{j=-n}^{n-1} (1+h \lambda^n_j)^{t h^{-1}}\big( e^n_j(0)- e^n_j(1/2n)\big) \int_{\T}\overline{e^n_j(\rho_n(y))} g(\rho_n(y))\,dy   
\,\Big|
\\
&\lesssim   n^{-1}\|g\|_{L_1(\Pi_n)}\sum_{j=-n}^{n-1}|1+h  \lambda^n_j| ^{t h^{-1} }|j|.
\end{equs}
Using \eqref{eq:Pn-exponential} and then Proposition \ref{prop:summation} (with $\lambda=\delta$ and $\gamma=1$), we get
\begin{equs}      \label{eq:to-be-applied-soon}
  | \cP_{t}^{n} g (0) -\cP_{t}^{n}  g (1/2n) |   \lesssim t^{-1} n^{-1}\|g\|_{L_1(\Pi_n)}.
\end{equs}
Let us   $\tilde{K}= \sqrt{2 \log(2n) /c}$ and let us consider two cases, namely,  $2  \sqrt{t} \tilde{K} \geq 1/4$  and  $2  \sqrt{t} \tilde{K} < 1/4$.

In the first case, the claim follows directly  by applying  \eqref{eq:to-be-applied-soon} to our given function $f$, since $\|f\|_{L_1(\Pi_n)} \leq \|f\|_{\bB(\T)}$. 

 We now focus to the case $2  \sqrt{t} \tilde{K} < 1/4$.
As a brief detour, take some $K>0$ and recall the notations $X_i$, $S_n$, $\widehat{S}^n_{ t}$ from Remark \ref{rem:random-walk} and the representation \eqref{eq:disc-semigroup}.
Denote furthermore by $N_n$ the number of nonzero elements in $\{X_1,\ldots,X_n\}$. Then conditionally on $N_n=\ell$, $(S_n+\ell)/2$ has binomial distribution with parameters $1/2$ and $\ell$.
Hoeffding's inequality implies that 
\begin{equ}
\bP\big(\Binom(1/2,\ell)\geq (\ell+K\sqrt{\ell})/2\big)\leq e^{-K^2/2},
\end{equ}
which gives
\begin{equs}
\bP\big(S_n\geq K\sqrt{n}\big)&=\sum_{\ell=1}^n\bP\big(S_n\geq K\sqrt{n}|N_n=\ell\big)\bP(N_n=\ell)
\\
&\leq\sum_{\ell=1}^n\bP\Big((S_n+\ell)/2\geq (\ell +K\sqrt{\ell})/2 \Big|N_n=\ell\Big)\bP(N_n=\ell)
\\
&\leq e^{-K^2/2}.
\end{equs}
Since $K$ and $n$  were arbitrary, this implies that 
\begin{equ}
\bP\big(\widehat{S}^n_{ t}\geq \tilde{K} \sqrt{t}\big)=\bP\big(S_{c^{-1}(2n)^2 t}\geq \tilde{K}\sqrt{(2n)^2 t}\big)\leq e^{-c\tilde{K}^2/2}.     \label{eq:exp-bound}
\end{equ}
Without loss of generality, we can assume that $n \geq 3$, in which case we have that  $\tilde{K}>c^{-1/2}$.  For $f : \Pi_n \to \R$,  by using the  representation \eqref{eq:disc-semigroup},   we have 
\begin{equs}
 | \cP_{t}^{n}  f (0) -\cP_{t}^{n}  f (1/2n) | &  = | \bE \tilde{f}( \widehat{S}^n_{t}) - \bE \tilde{f}(1/2n+ \widehat{S}^n_{t}) | 
 \\
& \leq   \big| \bE (\tilde{f}\bone_{A_{\tilde{K}, t}})( \widehat{S}^n_{t})  -\bE (\tilde{f}\bone_{A_{\tilde{K}, t}})(1/2n+ \widehat{S}^n_{t}) |
 \\
  & + \| f\|_{\bB(\Pi_n) }\Big( \bP\big( \widehat{S}^n_{t}  \in A_{\tilde{K}, t}^c  \big) + \bP\big( 1/2n+\widehat{S}^n_{t}  \in A^c_{\tilde{K}, t}\big)  \Big),
  \label{eq:less-srt-t}
\end{equs}
where $A_{\tilde{K}, t}= \{ x \in \R :  |x- k | \leq 2 \tilde{K} \sqrt{t}, \text{for some $k \in \mathbb{Z}$} \}$.  Notice that 
\begin{equs}
\bP \Big( 1/2n+\widehat{S}^n_{t}  \in A^c_{\tilde{K}, t}  \Big)
& \leq 
 \bP \Big( |1/2n+\widehat{S}^n_{t} |\geq  2\tilde{K}\sqrt{t}  \Big)  
 \\
 &   \leq \bP \Big(  |\widehat{S}^n_{t} |\geq  2\tilde{K}\sqrt{t} -1/2n  \Big) 
\\
&= \bP\Big( |\widehat{S}^n_{t} |\geq  2\tilde{K}\sqrt{t} -\sqrt{c^{-1}h} \Big) 
\\
& \leq  \bP( |\widehat{S}^n_{t} |\geq  \tilde{K}\sqrt{t} )    \leq 2 e^{-c\tilde{K}^2/2} ,
\end{equs}
where for the third inequality we have used that  $t \geq h $ and that $\tilde{K}>c^{-1/2}$,  and for the last step we have used \eqref{eq:exp-bound}.  The same bound holds for $\bP\big( \widehat{S}^n_{t}  \in A_{\tilde{K}, t}^c  \big)$, so that 
\begin{equs}
\bP\big( \widehat{S}^n_{t}  \in A_{\tilde{K}, t}^c  \big) + \bP\big( 1/2n+\widehat{S}^n_{t}  \in A^c_{\tilde{K}, t}\big) \leq 4e^{-c\tilde{K}^2/2}.  \label{eq:whatever}
 \end{equs}
 On the other hand,  to estimate the first term on the right hand side of \eqref{eq:less-srt-t}, we can use \eqref{eq:to-be-applied-soon} with $g(\cdot)= (\tilde{f} \bone_{A_{\tilde{K}, t}})(\cdot)$. This gives 
 \begin{equs}
  \big| \bE (\tilde{f}\bone_{A_{\tilde{K}, t}})( \widehat{S}^n_{t})  -\bE (\tilde{f}\bone_{A_{\tilde{K}, t}})(1/2n+ \widehat{S}^n_{t}) |
& \lesssim  {t}^{-1} \frac{1}{2n} \|  f (\cdot)\bone_{|\cdot| \leq 2 \tilde{K} \sqrt{t} }\|_{L_1(\Pi_n)}
 \\
 & \lesssim {t}^{-1/2}  \frac{1}{n} \tilde{K} \|f \|_{\mathbb{B}(\Pi_n)}
 \end{equs}
 Consequently, by \eqref{eq:less-srt-t}, and the above, we get 
 \begin{equs}
  | \cP_{t}^{n}  f (0) -\cP_{t}^{n}  f (1/2n) |  \lesssim   {t}^{-1/2} \tilde{K} \frac{1}{n} \|f \|_{\mathbb{B}(\Pi_n)}+ e^{-c\tilde{K}^2/2} \| f \|_{\mathbb{B}(\Pi_n)},
\end{equs}
which upon recalling that $\tilde{K}= \sqrt{2 \log(2n) /c}$ gives
\begin{equs}
  | \cP_{t}^{n}  f (0) -\cP_{t}^{n}  f (1/2n) |  \lesssim   {t}^{-1/2} \sqrt{\ln(2n)} n^{-1} \|f \|_{\mathbb{B}(\Pi_n)}.
\end{equs}
As mentioned at the beginning of the proof, this yields \eqref{eq:discrete-HK-bounds} in the case $\alpha=0$, $\beta=1$. The case $\beta=1$, $\alpha=1$  follows from the trivial bound
$$
| \cP_{t}^{n}  f (x) -\cP_{t}^{n}  f (z) | \leq |x-z| \|f\|_{\C^1(\Pi_n)}.
$$
The case $\beta=1$, $\alpha\in[0,1]$ then follows by interpolation, see Remark \ref{rem:interpolation}.
Finally, the case $\beta \in [0, 1]$, $\alpha \in [0, \beta]$  follows by elementary interpolation (that is, $a\leq b$ and $a\leq c$ implying $a\leq b^\lambda c^{1-\lambda}$ for $\lambda\in[0,1]$) between \eqref{eq:discrete-HK-bounds} with $\beta =1$ and the trivial bound
\begin{equs}
| \cP_{t}^{n}  f (x) -\cP_{t}^{n}  f (z) | \leq |x-z|^\alpha \|f\|_{\C^\alpha(\Pi_n)}.
\end{equs}
This finishes the proof.  
\end{proof}

\begin{lemma}     \label{lem:preservation-reg-T}
Let $\alpha \in [0,1]$. There exists a constant $N(\alpha)$ such that for all $\psi \in \C^\alpha(\T)$, $t \in [0,1]$,  we have 
\begin{equs}
\|\cP^n_t \psi\|_{\C^\alpha(\T)} \leq N  \| \psi\|_{\C^\alpha(\T)}.
\end{equs}
\end{lemma}

\begin{proof}

One can easily see the estimate 
\begin{equs}      \label{eq:sup-sup}
\|\cP^n_t \psi\|_{\mathbb{B}(\T)} \leq N  \| \psi\|_{\mathbb{B}(\T)},
\end{equs}
so we focus on proving that for all $t \in [0,1]$ and $x, y \in \T$, we have 

\begin{equs}
|\cP^n_t \psi(x)- \cP^n_t \psi(y)| \leq N |x-y|^\alpha \| \psi\|_{\C^\alpha(\T)}.
\end{equs}
Let us first prove the claim with $\alpha=1$.  In addition, let us assume for now that $t \in \Lambda_n$. There are three cases: 

\emph{Case 1:  $|x-y| < 1/2n$ and $\rho_n(x)=\rho_n(y)$.} 
 In this case, by \eqref{eq:linear-interpolation} it follows that 
\begin{equs}
|\cP^n_t \psi (x)-\cP^n_t \psi (y)| & = \big|2n (x-y) \big( \cP^n_t \psi (\rho_n(x)+1/2n)- \cP^n_t \psi (\rho_n(x)) \big) \big|
\\
  & \lesssim  |x-y| \|\psi\|_{\C^1(\T)},
\end{equs}
where we have used the representation \eqref{eq:disc-semigroup}. 

\emph{Case 2:  $|x-y| < 1/2n$ and $\rho_n(x) \neq \rho_n(y)$.} In this case, let us assume without loss of generality that $\rho_n(y)=  \rho_n(x)+1/2n$. By \eqref{eq:linear-interpolation}, we see that 
\begin{equs}
\cP_t^n \psi (x)= \cP_t^n \psi ( \rho_n(x) ) + 2n (x-\rho_n(x)) \big( \cP_t^n \psi ( \rho_n(x)+1/2n ) - \cP_t^n \psi ( \rho_n(x) ) \big),
\end{equs} 
which implies that 
\begin{equs}
 | \cP_t^n \psi (x)- \cP_t^n \psi ( \rho_n(y) ) |  
&= | \cP_t^n \psi (x)- \cP_t^n \psi ( \rho_n(x) +1/2n) |  
\\
 &=\big|  \big(2n (x-\rho_n(x))-1  \big)  \big( \cP_t^n \psi ( \rho_n(x)+1/2n ) - \cP_t^n \psi ( \rho_n(x) ) \big) \big|
\\
& \leq | x- \rho_n(x) -1/2n| \|\psi\|_{\C^1(\T)}
\\
&= |x- \rho_n(y)| \| \psi\|_{\C^1(\T)}. 
\end{equs}
We also have 
\begin{equs}
| \cP_t^n \psi ( \rho_n(y) ) -\cP_t^n \psi (y)|  &= \big|  \big(2n (y-\rho_n(y)) \big)  \big( \cP_t^n \psi ( \rho_n(y)+1/2n ) - \cP_t^n \psi ( \rho_n(y ) \big) \big|
\\
& \leq  |y-\rho_n(y) |  \| \psi\|_{\C^1(\T)}. 
\end{equs}
Consequently, 
\begin{equs}
| \cP_t^n \psi (x) -\cP_t^n \psi (y)|  \leq&   | \cP_t^n \psi ( x) -\cP_t^n \psi ( \rho_n(y) )|  + | \cP_t^n \psi ( \rho_n(y) ) -\cP_t^n \psi (y)|  
\\
\leq & \big( |x- \rho_n(y)|+|y-\rho_n(y) | \big) \| \psi\|_{\C^1(\T)}
\\
=& |x-y| \| \psi\|_{\C^1(\T)}.
\end{equs}

\emph{Case 3:  $|x-y| \geq  1/2n$}. In this case, we have 
\begin{equs}
&| \cP_t^n \psi (x) -\cP_t^n \psi (y)| 
\\
&\leq  | \cP_t^n \psi (x) -\cP_t^n \psi (\rho_n(x))| +| \cP_t^n \psi (\rho_n(x)) -\cP_t^n \psi (\rho_n(y))| +| \cP_t^n \psi (y) -\cP_t^n \psi (\rho_n(y))| 
\\
& \lesssim   \  \big( 1/2n+ | \rho_n(x)-\rho_n(y) |\big)\| \psi\|_{\C^1(\T)}
\\
& \lesssim  |x-y| \| \psi\|_{\C^1(\T)},
\end{equs}
where we have used the results from the case $|x-y|< 1/2n$ for the first and the third term, the representation \eqref{eq:disc-semigroup} for the second term, and of course the fact that $|x-y| > 1/2n$.

Hence, we have proved the claim for $t \in \Lambda_n$. If $t \geq 0$,  then the claim follows from the case  $t \in \Lambda_n$ virtue of the identity  
\begin{equs}
 \cP_t^n \psi (x)  = \cP_{\kappa_n(t)}^n \psi (x) +h^{-1}(t-\kappa_n(t)) \big(  \cP_{\kappa_n(t)+h}^n \psi (x) - \cP_{\kappa_n(t)}^n \psi (x) \big),
\end{equs}
see
\eqref{eq:def-of-mu}.

To summarise, we have shown the desired inequality with $\alpha =1$ for all $t \geq 0$ and $x, y \in \T$. 
From \eqref{eq:sup-sup}, it also follows that
\begin{equs}
| \cP_t^n \psi (x) -\cP_t^n \psi (y)|  \lesssim \| \psi\|_{\mathbb{B}(\T)},
\end{equs}
for all $t \geq 0$ and $x, y \in \T$, which is the claim for $\alpha=0$.  Finally,  the case $\alpha \in (0,1)$ follows by interpolation, see Remark \ref{rem:interpolation}.
\end{proof}

\begin{lemma}     \label{lem:terms-from-IC}
Let $\beta \in [0, 1]$, $ \alpha \in [0, \beta]$. The following hold: 
\begin{enumerate}
\item There  exists  a constant $N$ such that for all $\psi \in \C^\alpha(\T)$,  $n  \in \mathbb{ N}$, $r \geq h $, $y \in \T$,  $z\in\{y,\rho_n(y)\}$
\begin{equs}        \label{eq:disc-HK-applied}
|\cP_{r }^{n} \psi(y)- 
\cP_{\kappa_n(r)}^{n} \psi(z)| &  \leq N \big(\log(2n) \big)^{(\beta-\alpha)/2} n^{-\beta} r^{(\alpha- \beta)/2} \| \psi\|_{\C^\alpha(\T)}.
\end{equs} 

\item There exists  a constant $N$ such that for all $\psi \in \C^\alpha (\T)$,  $n  \in \mathbb{ N}$, $r \in [0, h] $, $y \in \T$, $z\in\{y,\rho_n(y)\}$
\begin{equs}        \label{eq:disc-HK-near-zero}
|\cP_{r }^{n} \psi(y)- 
\cP_0^{n} \psi(z)| &  \leq N n^{-\alpha}  \| \psi\|_{\C^\alpha(\T)}.
\end{equs} 
\end{enumerate}
\end{lemma}
\begin{proof}
We have 
\begin{equs}
|\cP_{r }^{n} \psi(y)- 
\cP_{\kappa_n(r)}^{n} \psi(\rho_n(y))| &  \leq|\cP_{r }^{n} \psi(y)- \cP_{\kappa_n(r)}^{n} \psi(y)|
\\
&\qquad +
|\cP_{\kappa_n(r)}^{n} \psi(y)- \cP_{\kappa_n(r)}^{n} \psi(\rho_n(y))|.
\end{equs} 
For the second term, keeping in mind the definition of $e^n$ in \eqref{eq:linear-interpolation}, we see that 
\begin{equs}
 |\cP_{\kappa_n(r)}^{n} &\psi(y)- \cP_{\kappa_n(r)}^{n} \psi(\rho_n(y))|
\\
&=  \Big| 2n \big(y-\rho_n(y)\big) \Big( \cP^{n}_{\kappa_n(r)} \psi \big( \rho_n(y)+ (2n)^{-1}\big) - \cP^{n}_{\kappa_n(r)} \psi ( \rho_n(y)) \Big)  \Big| 
\\
&\lesssim \big(\log(2n) \big)^{(\beta-\alpha)/2} n^{-\beta} \big(\kappa_n(r) \big)^{(\alpha- \beta)/2} \| \psi\|_{\C^\alpha(\T)} \label{eq:idk50}
\end{equs}
where we have used  Lemma \ref{lem:Discrete-HK-bounds} for the inequality.  Similarly, by \eqref{eq:def-of-mu}, we see that 
\begin{equs}
|\cP_{r }^{n} \psi(y)- 
\cP_{\kappa_n(r)}^{n} \psi(y)| 
&
= h^{-1} (r-\kappa_n(r))   | \cP_{\kappa_n(r)+h} \psi (y)- \cP_{\kappa_n(r)} \psi (y)| 
\\
& \leq | \cP_{\kappa_n(r)+h} \psi (y)- \cP_{\kappa_n(r)} \psi (y)| 
\\
& \lesssim   | \cP_{\kappa_n(r)+h} \psi (y)- \cP_{\kappa_n(r)+h} \psi ( \rho_n(y))| 
\\
 &\quad +| \cP_{\kappa_n(r)+h} \psi (\rho_n(y) )- \cP_{\kappa_n(r)} \psi ( \rho_n(y))| 
 \\
  &\quad +| \cP_{\kappa_n(r)} \psi (\rho_n(y) )- \cP_{\kappa_n(r)} \psi ( y)| 
\end{equs}
The first and the third term at the right hand side can be bounded by the right hand side of \eqref{eq:disc-HK-applied} because of \eqref{eq:idk50}.  For the second term, we have by \eqref{eq:P^n-rep}  that 
\begin{equs}
 | \cP_{\kappa_n(r)+h} \psi (\rho_n(y))- \cP_{\kappa_n(r)} \psi (\rho_n(y))|  = | h \Delta_n \cP_{\kappa_n(r)} \psi (\rho_n(y)) | , 
 \end{equs}
 which again can be estimated by the right hand side of \eqref{eq:disc-HK-applied} by virtue of Lemma \ref{lem:Discrete-HK-bounds}. This shows that 
 \begin{equs}
 |\cP_{r }^{n} \psi(y)- 
\cP_{\kappa_n(r)}^{n} \psi(y)|  \lesssim  \big(\log(2n) \big)^{(\beta-\alpha)/2} n^{-\beta} r^{(\alpha- \beta)/2} \| \psi\|_{\C^\alpha(\T)},
 \end{equs}
 which proves \eqref{eq:disc-HK-applied} with $z = y $ and combined with \eqref{eq:idk50} gives \eqref{eq:disc-HK-applied} with $z=\rho_n(y)$. 
 
 The proof of \eqref{eq:disc-HK-near-zero} is more straightforward and is left to the reader. 
\end{proof}

The following lemma, a key error estimate between the continuous and the discrete heat kernels, is similar to \cite[Lemma~3.3]{Gy}, where an estimate of the form $\eqref{eq:Pn-P}$ is proved for the Dirichlet setting.  Our version is a bit more flexible by allowing $\beta \in [0, 2]$ in contrast to \cite[Lemma~3.3]{Gy} where  $\beta \in [0,1]$. 
\begin{lemma}   \label{lem:Pn-P}
Let $\beta \in [0, 2]$. Then there exists a constant $N(\beta , c)$ such that 
for all $t \in  [h, 1]$, $x \in \T$ one has the bound
\begin{equs}\label{eq:Pn-P}
\| p_t(x-\cdot)-p^{n}_{\kappa_n(t)}(x,\cdot) \|^2_{L_2(\T)} \leq  N n^{-\beta} t^{-(\beta+1)/2}.
\end{equs} 
\end{lemma}
\begin{proof}
Note that since $t\gtrsim n^{-2}$,  it suffices to prove \eqref{eq:Pn-P} in the $\beta=2$ case.
Let us start by defining
\begin{equs}
\tilde p^{n}_t(x,y) = \sum_{j=-n}^{n-1} e^{\lambda^n_jt} e^n_j(x) e_j(\rho_n(y)).
\end{equs}
Writing
\begin{equ}
\|p_t(x-\cdot)-p^{n}_t(x,\cdot)\|_{L_2(\T)} ^2 \leq 2 \|p_t(x-\cdot)-\tilde p^n_t(x,\cdot)\|_{L_2(\T)} ^2 +2 \|\tilde p^n_t(x,\cdot)-p^{n}_{\kappa_n(t)}(x,\cdot)\|_{L_2(\T)} ^2,
\label{eq:Pn-P-triangle}
\end{equ}
we start with an estimate for the first term at the right hand side of \eqref{eq:Pn-P-triangle}.  We write
\begin{equs}
\|p_t(x-\cdot)-\tilde p^n_t(x,\cdot)\|_{L_2(\T)} ^2 \leq 4 \sum_{i=1}^4 I^i_t (x),
\end{equs}
with 
\begin{equs}
I^1_t(x) & = \| \sum _{ \{ j  \leq -n-1\}\ \cup  \{ j \geq n\}} e^{\lambda_j t } e_j(x) \overline{ e_j }(\cdot)\| _{L_2(\T)}^2,
\\
I^2_t(x) & = \| \sum _{ -n \leq j \leq n-1}(  e^{\lambda_j t }-e^{\lambda^n _j t } ) e_j(x) \overline{ e_j }(\cdot) \| _{L_2(\T)}^2,
\\
I^3_t(x) & = \| \sum _{ -n \leq j \leq n-1}e^{\lambda^n _j t } \big( e_j(x) -e^n_j(x)\big)  \overline{ e_j } (\cdot)\| _{L_2(\T)}^2,
\\
I^4_t(x) & = \|  \sum _{ -n \leq j \leq n-1} e^{\lambda^n _j t } e^n_j(x) \big(  \overline{ e_j } - \overline{e_j \circ \rho_n}  \big)(\cdot) \| _{L_2(\T)}^2.
\end{equs}
We now show that each $I^i_t(x)$ is bounded by the right-hand side of \eqref{eq:Pn-P}.
By the orthonormality of $e_j$ we see that
\begin{equs}
I^1_t(x) & =    \sum _{ \{ j  \leq -n-1\}\ \cup  \{ j \geq n\}} e^{2\lambda_j t } 
 \lesssim n^{-2} \sum _{|j| \geq 2n } j^2 e^{2\lambda_j t }.
   \label{eq:estimate-I1}
\end{equs}
Since $\lambda_j=-4\pi^2j^2$, we can use Proposition \ref{prop:summation} to conclude the claimed bound.
Next recall that $\lambda_j \leq \lambda^n_j \leq 0$ and that $\lambda^n_j \leq -16j^2$ by \eqref{eq:sunday}, we get that
\begin{equs}
I^2_t(x)  & \leq \sum _{ -n \leq j \leq n-1} \big( e^{\lambda_j t }  -  e^{\lambda^n _j t }  \big)^2= \sum _{ -n \leq j \leq n-1}   \big(  e^{\lambda^n_j t }  ( 1-  e^{-(\lambda^n _j-\lambda_j) t } ) \big)^2 
\\
&  \leq \sum_{ -n \leq j \leq n-1} \big(  e^{\lambda^n_j t }  (\lambda^n_j-\lambda_j) t\big)^2  
\lesssim  \sum_{ -n \leq j \leq n-1}  e^{-32j^2  t }   j^4 (\gamma ^n_j-1)^2 t^2.
\end{equs}
By \eqref{eq:bound1-gamma} and Proposition \ref{prop:summation} we get
\begin{equs}
I^2_t(x) \lesssim \sum_{ -n \leq j \leq n-1}  e^{-32j^2  t } j^8 n^{-4} t^2&  \lesssim n ^{-4} t^{2-9/2}.
  \label{eq:estimate-I2}
\end{equs}
Using $n^{-2}\lesssim t$ again, we get the required bound. Next, for $I^3_t(x) $ we can use that 
\begin{equs}
| e_j(x) -e^n_j(x) |&\leq| e_j(x) -e_j(\rho_n(x)) |+| e_j^n(\rho_n(x)) -e^n_j(x) |\leq n^{-1}\big(\|e_j\|_{\cC^1}+\|e_j^n\|_{\cC^1}\big)\lesssim jn^{-1}.
\end{equs}
Therefore,
\begin{equs}
I^3_t(x) &  =   \sum_{ -n \leq j \leq n-1}e^{2\lambda^n _j t } | e_j(x) -e^n_j(x) |^2   \lesssim \sum_{ -n \leq j \leq n-1} e^{-32j^2 t }  j^2 n^{-2}.
 \label{eq:estimate-I3}
\end{equs}
As usual, Proposition \ref{prop:summation} implies the claimed bound.
Before estimating $I^4_t(x)$, we claim that 
 for  $j, \ell \in \{-n, ..., n-1\}$, with $j \neq \ell$, the functions $e_j-e_j \circ \rho_n$ and $e_\ell-e_\ell \circ \rho_n $ are orthogonal in $L_2(\T)$. Indeed, by the orthogonality of $(e_j)_{j \in \mathbb{Z}}$ and the orthogonality of $(e^n_\ell)_{-n \leq \ell \leq n-1 }$, we have that 
 \begin{equs}
 ( e_j-e_j \circ \rho_n, e_\ell-e_\ell\circ \rho_n)_{L_2(\T)} = -  ( e_j, e_\ell\circ \rho_n)_{L_2(\T)}-  ( e_j \circ \rho_n, e_\ell)_{L_2(\T)}.
 \end{equs}
 Then we see that for $j, n \in \{ -n, ..., n-1\}$, $j \neq \ell$, we have 
 \begin{equs}
 ( e_j, e_\ell\circ \rho_n)_{L_2(\T)}& = \sum_{k=0}^{2n} \int_{k/2n}^{(k+1)/2n} e^{i 2\pi  j y }  e^{-i2\pi  \ell k/2n} \, dy  = \frac{(e^{i 2\pi j /2n} -1)}{i 2 \pi  j} \sum_{k=0}^{2n-1}  e^{i 2\pi  (j-\ell) k/2n } 
 \\
 & = \frac{(e^{i 2\pi j /2n} -1)}{i 2 \pi  j} \frac{ (1-  e^{i 2\pi  (j-\ell) }) }{(1- e^{i 2\pi  (j-\ell)/2n  }) } =0,
 \end{equs}
 which shows the claim. Consequently, 
\begin{equs}
I^{4 }_t(x)  
&  =  \sum _{ -n \leq j \leq n-1 } e^{2\lambda^n _j t } \| e_j-e_j \circ \rho_n\| _{L_2(\T)}^2 \lesssim\sum_{ -n \leq j \leq n-1} e^{-32j^2 t } j^2n^{-2},
 \label{eq:estimate-I4}
\end{equs}
giving the claimed bound as before. Putting \eqref{eq:estimate-I1}-\eqref{eq:estimate-I4} together, we conclude that 
\begin{equs}
\|p_t(x-\cdot)-\tilde p^n_t(x,\cdot)\|_{L_2(\T)} ^2 \lesssim n^{-\beta } t^{-(1+\beta)/2}. \label{eq:estimate-P-Rn}
\end{equs}
Next we see that 
\begin{equs}
\|\tilde p^n_t(x,\cdot)-p^{n}_{\kappa_n(t)}(x,\cdot)\|_{L_2(\T)} ^2 &  \lesssim  \sum_{-n \leq j \leq n-1} | e^{\lambda^n_j t} - e^{\lambda^n_j \kappa_n(t)} |^ 2
\\
& \qquad + \sum_{-n \leq j \leq n-1} | e^{\lambda^n_j \kappa_n (t) } - (1+ h \lambda^n_j)^{\kappa_n(t) h^{-1}}  |^ 2.\label{eq:another-triangle}
\end{equs}
As before, we show that both terms are bounded by the right-hand side of \eqref{eq:Pn-P}.
Since $t\geq h$ implies $\kappa_n(t)\geq t/2$, we have
\begin{equs}
 \sum_{-n \leq j \leq n-1} | e^{\lambda^n_j t} - e^{\lambda^n_j \kappa_n(t)} |^ 2   & \leq   \sum_{-n \leq j \leq n-1} \big(  |t-\kappa_n(t)|\lambda^n_j e^{\lambda^n_j \kappa_n(t)} \big) ^2 
 \\
  & \lesssim   n^{-4} \sum_{-n \leq j \leq n-1} j^4 e^{-16j^2 t}.
    \label{eq:kernel-time-dis-1}
\end{equs}
Proposition \ref{prop:summation} and $n^{-2}\lesssim t$ yields the claimed bound.
For the other term,  we have 
 \begin{equs}      
\sum_{-n \leq j \leq n-1} | e^{\lambda^n_j \kappa_n (t) } - (1+ h \lambda^n_j)^{\kappa_n(t) h^{-1}} |^ 2 & \lesssim    \sum_{j \in J^1_n } | e^{\lambda^n_j \kappa_n (t) }  |^ 2   
+ \sum_{ j \in J^1_n } | (1+ h \lambda^n_j)^{\kappa_n (t) h^{-1}} |^ 2 
\\
&\quad +  \sum_{ j \in J^2_n } | e^{\lambda^n_j \kappa_n (t) } - (1+ h \lambda^n_j)^{\kappa_n(t) h^{-1}} |^ 2   ,         \label{eq:idk-how-to-call-it}
\end{equs}
where 
\begin{equs}
J^1_n & = \{ j \in \{-n , ..., n-1\} : h \lambda^n_j \in [-( 4c \vee 1/2), -1/2] \}, 
\\
 J^2_n & = \{ j \in \{-n , ..., n-1\} : h \lambda^n_j \in (  -1/2, 0] \}. 
\end{equs}
We have, using $\kappa_n(t)\geq t/2$ as before,
\begin{equs}
 \sum_{j \in J^1_n } | e^{\lambda^n_j \kappa_n (t) }  |^ 2 
 \leq \sum_{j \in J^1_n }e^{-t/(2h)}\lesssim n(t/h)^{-3/2}\lesssim n^{-2}t^{-3/2},
\end{equs}
and similarly, by \eqref{eq:Pn-exponential},
\begin{equs}
 \sum_{ j \in J^1_n } | (1+ h \lambda^n_j)^{\kappa_n (t) h^{-1}} |^ 2 \lesssim   n^{-2} t  ^{-3/2}.
\end{equs}
For the last term on the right-hand side of \eqref{eq:idk-how-to-call-it} we use the elementary inequalities
$0 \leq x-\ln(1+x) \leq x^2$ for all $x \in [-1/2, 0]$,
and $|1-e^y|\leq |y|$ for all $y\leq 0$. Therefore,
\begin{equs}
 \sum_{ j \in J^2_n } | e^{\lambda^n_j \kappa_n (t) } - (1+ h \lambda^n_j)^{\kappa_n (t) h^{-1}} |^ 2      
&=   \sum_{ j \in J^2_n } e^{2\lambda^n_j \kappa_n (t) } \Big| 1 - e^{\kappa_n (t) h^{-1}\big( \ln (1+ h \lambda^n_j) - \lambda^n_j h  \big) } \Big|^ 2
\\
&\leq   \sum_{ j \in J^2_n }  e^{-16j^2 t }|  \kappa_n (t) h^{-1} ( { h \lambda^n_j- \ln (1+ h \lambda^n_j) } ) |^ 2
\\
&\leq    \sum_{ j \in J^2_n }  e^{-16j^2t} |\kappa_n (t)  h^{-1}|^2 |h \lambda^n_j|^4  
\\
&\lesssim   n^{-4 } |t|^2 \sum_{ j \in\Z}    e^{-16j^2t }  j ^8.
\end{equs}
As before, this gives the desired bound.
We conclude that 
 \begin{equs}      
\sum_{-n \leq j \leq n-1} | e^{\lambda^n_j \kappa_n (t) } - (1+ h \lambda^n_j)^{ \kappa_n(t) h^{-1}} |^ 2 \lesssim    n^{-2 } t ^{-3 /2}.\label{eq:last-bit}
\end{equs}
Combining \eqref{eq:Pn-P-triangle}, \eqref{eq:estimate-P-Rn}, \eqref{eq:another-triangle}, \eqref{eq:kernel-time-dis-1}, and \eqref{eq:last-bit} brings the proof to an end.
\end{proof}

For the proof of the Lemma \ref{lem:det-rate} below, let us recall the following result, a nice consequence of Stein's method of normal approximation (see, for example, \cite[Corollary 4.2, p. 68]{Stein}).
\begin{lemma}        \label{lem:stein}
Let $\{Y_i\}_{i=1}^\infty$ be i.i.d. random variables with $\E Y_1=0$ and $\E Y_1^2=1$ and let $W\sim \mathcal{N}(0,1)$. Then, for all $f\in\cC^1(\R)$ and $k \in \mathbb{N}$  we have 
\begin{equs}
| \E f(W_k) - \E f(W) | \leq k^{-1/2}  \| f \|_{\C^1} \E |Y_1|^3 ,
\end{equs}
where $W_k = k^{-1/2}\sum_{i=1}^k Y_i$.

\end{lemma}

\begin{lemma}     \label{lem:det-rate}
For any  $\alpha \in [0, 1]$ there exists a constant $N=N(\alpha,c)$ such that for all $\psi \in \cC^\alpha(\T)$, $t \in [0,1]$, and  $y \in \T$, we have 
\begin{equs}
| \cP^n_t \psi(y) - \cP_t \psi (y) | \leq N n^{-\alpha} \| \psi \|_{\C^\alpha(\T)}.    \label{eq:rate-deterministic}
\end{equs}
\end{lemma}

\begin{proof}
We first show the desired inequality for $\alpha=1$. We use the reformulation of the discrete heat kernel as the transition kernel of a discrete time random walk, see Remark \ref{rem:random-walk}.

Further, we assume for now that $t \in \Lambda_n$, $y\in \Pi_n$. Recall that in this case, by $\eqref{eq:disc-semigroup}$ we have that
\begin{equs}
| \cP^n_t \psi(y) - \cP_t \psi (y) | &= |  \E \tilde \psi (y+\widehat{S}^n_ { t})- \E \tilde \psi (y+\sqrt{2} W_ { t})|
\\
&= \Big|  \E \tilde \psi \Big(y+\sqrt{2t} \sum_{i-1}^{h^{-1}t} \frac{X_i}{\sqrt{2c} \sqrt{h^{-1} t}} \Big)- \E \tilde \psi (y+\sqrt{2  t } W_ { 1}) \Big|
\end{equs}
where $W_t$ is a Brownian motion. By applying Lemma \ref{lem:stein} with $Y_i= X_i/ \sqrt{2c}$ and $f(\cdot)= \tilde{\psi} (y+\sqrt{2t} \ \cdot)$, we get 
\begin{equs}
| \cP^n_t \psi(y) - \cP_t \psi (y) | \lesssim (h^{-1}t)^{-1/2} t^{1/2} \| \psi \|_{\C^1(\T)} \leq N n^{-1} \| \psi \|_{\C^1(\T)}.
\end{equs}
For $t \in [0,1]$, $y \in \T$, we have 
\begin{equs}
| \cP^n_t \psi(y) - \cP_t \psi (y) | &  \leq  | \cP^n_t \psi(y) - \cP^n_{\kappa_n(t)} \psi (\rho_n(y)) |
+| \cP^n_{\kappa_n(t)} \psi (\rho_n(y)) - \cP_{\kappa_n(t)} \psi (\rho_n(y)) |
\\
&\qquad+| \cP_{\kappa_n(t)} \psi (\rho_n(y)) - \cP_t \psi (y) |
\\
& \lesssim n^{-1} \| \psi \|_{\C^1(\T)},
\end{equs}
where for the first term we have used Lemma \ref{lem:terms-from-IC} ( \eqref{eq:disc-HK-applied} for $t\geq h$ and \eqref{eq:disc-HK-near-zero} for $t \in [0,h]$) and for the last term we have used classical heat kernel estimates (for example \eqref{eq:HK bound mixed} with  $\alpha= \beta =1$).  This proves the claim for $\alpha=1$. Also, for $\alpha=0$ the claim trivially holds. Finally, the claim for $\alpha \in (0,1)$ follows by interpolation, see Remark \ref{rem:interpolation}.
\end{proof}

\subsection{Discrete and continuous Ornstein-Uhlenbeck processes}
The last integral in \eqref{eq:main in mild form} is also called the infinite dimensional Orstein-Uhlenbeck process. In the trivial case $\psi=0$, $b=0$, that is simply the solution process, and it also plays an important role in the general case.
Let us therefore introduce a separate notation for it:
\begin{equ}
O_t(x)=\int_0^t\int_{\T} p_{t-s}(x-y)\,\xi (dy, ds) , \qquad (t,x) \in [0,1] \times \T. 
\end{equ}
Let $(s,t)\in[0,1]_<$.Thanks to the semigroup property of $\cP$ the Ornstein-Uhlenbeck process satisfies
\begin{equs}
O_t(x) & =\int_0^s\int_\T p_{t-r}(x-y)\,\xi (dy, dr) 
+\int_s^t\int_\T p_{t-r}(x-y)\,\xi (dy, dr) 
\\
&=\int_0^s\int_\T\int_\T p_{t-s}(x-z)p_{s-r}(z-y)\,dz\,\xi (dy, dr) 
+\int_s^t\int_\T p_{t-r}(x-y)\,\xi (dy, dr) 
\\
&=\big(\cP_{t-s} O_s\big)(x)+\int_s^t\int_\T p_{t-r}(x-y)\,\xi (dy, dr) .\label{eq:OU transition}
\end{equs}
The second term on the right-hand side of is a Gaussian random variable that is independent of $\cF_s$. Its variance is given by
\begin{equs}\label{eq:Q-def}
\int_s^t\int_\T|p_{t-r}(x-y)|^2\,dy\,dr=\int_0^{t-s}\int_\T |p_{r}(y)|^2\,dy\,dr=:Q(t-s).
\end{equs}
Therefore, for any  bounded measurable $g: \R \to \R$ and $\cF_s$-measurable random variable $Y$ one has the almost sure equality
\begin{equ}\label{eq:OU-transition-main}
\E^sg\big(O_t(x)+Y\big)=\cP_{Q(t-s)}^\R g\big(\big(\cP_{t-s} O_s\big)(x)+Y\big).
\end{equ}
Moreover, 
from \eqref{eq:HK-L2} one gets with an absolute constant $N>0$ for all $r\in(0,1]$ the bound
\begin{equ}\label{eq:Q-easy-bound-first}
N^{-1} \sqrt{r}\leq  Q(r)\leq N\sqrt{r}.
\end{equ}
The following is well-known.
\begin{proposition}\label{prop:OU Holder}
For any $p>0$, and $\theta \in(0,1/2)$ one has $\|O\|_{L_p(\Omega;\C([0,1];\C^{1/2-\theta}(\T))}<\infty$.
\end{proposition}
We similarly define the discrete Ornstein-Uhlenbeck process by setting
\begin{equ}
O_t^n(x)=\int_0^t\int_\T p_{\kappa_n(t-r)}^{n}(x,y)\,\xi (dy, dr), \qquad (t,x)\in [0,1] \times \T.  
\end{equ}
We wish to obtain a discrete analogue of \eqref{eq:OU-transition-main}.
For $s\in[0,t]$ we write
\begin{equs}
O^n_t(x)&=\int_0^s\int_\T p_{\kappa_n(t-r)}^{n}(x,y)\,\xi (dy, dr) +\int_s^t\int_\T p_{\kappa_n(t-r)}^{n}(x,y)\,\xi (dy, dr)
\\
&=:\widehat O^n_{s,t}(x)+\int_s^t\int_\T p_{\kappa_n(t-r)}^{n}(x,y)\,\xi (dy, dr).
\end{equs}
Clearly, $\widehat O^n_{s,t}(x)$ is $\cF_s$-measurable, while the second term in the right-hand side is a Gaussian random variable that is independent of $\cF_s$.
Its variance is given by
\begin{equs}
Q^n(t-s)  : =\int_0^{t-s}\int_\T |p^{n}_{\kappa_n(r)}(x,y)|^2\,dy\,dr = \int_0^{t-s} \sum_{j=-n}^{n-1} |1+h \lambda^n_j|^{  2 \kappa_n(r) h^{-1} } \, dr.
\end{equs}
Therefore, for any continuous function $g$ and $\cF_s$-measurable random variable $Y$ one has the almost sure equality
\begin{equ}\label{eq:OUn-transition-main}
\E^sg\big(O_t^n(x)+Y\big)=\cP_{Q^n(t-s)}^\R g\big(\widehat O^n_{s,t}(x)+Y\big).
\end{equ}
In the next two statements we compare the expressions in \eqref{eq:OUn-transition-main} to the corresponding ones in \eqref{eq:OU-transition-main}.
\begin{corollary}
Let $\beta\in[0,2]$ and $p>0$. Then there exists a constant $N(p, \beta, c)$ such that for all $t, s \in [0,1] $ with $t-s \geq h$, $x\in\T$ one has
\begin{equs}                            \label{eq:Estimate-BDG}
\|\cP_{t-s}O_s(x)-\widehat{O}^n_{s,t}(x)\|_{L_p(\Omega)}
\leq N n^{-\beta/2 } |t-s|^{ (1-\beta)/4}. 
\end{equs}
Moreover, for $\beta \in [0,1)$, $p>0$, there exists a constant $N(p,\beta,c)$ such that for all $t\in[0,1]$, $x\in\T$, one has
\begin{equ}  \label{eq:Estimate-BDG2}
\|O_t(x)-O^n_t(x)\|_{L_p(\Omega)}\leq Nn^{-\beta/2} t^{(1-\beta)/4}.
\end{equ}
\end{corollary}
\begin{proof}
We will assume that $p=2$, since the general case follows by the equivalence of Gaussian moments. 
Notice that for $t-s \geq h$, the estimate \eqref{eq:Estimate-BDG} with $\beta=2$   follows directly by It\^o's isometry and Lemma \ref{lem:Pn-P}. Since it is true for $\beta =2$ is is also true for any $\beta \in [0,2]$ since $1/n \le 2c^{-1/2}  |t-s|^{1/2}$. 

As for \eqref{eq:Estimate-BDG2}, there are two cases.
First, assume that $t \in [0,h]$. By It\^o's isometry and \eqref{eq:Q-easy-bound-first}, we have 
\begin{equs}
\|O_t(x)\|_{L_2(\Omega)} = |Q(t)|^{1/2} \lesssim  t^{1/4} \lesssim n^{-\beta/2} t^{(1-\beta)/4}.
\end{equs}
By \eqref{eq:Qn-small-times} (below), we have 
\begin{equs}
\|O^n_t(x)\|_{L_2(\Omega)} = |Q^n(t)|^{1/2}=  (2 n t)^{1/2} \lesssim n^{-\beta/2} t^{(1-\beta)/4}.
\end{equs}
Combining the two estimate above gives \eqref{eq:Estimate-BDG2}. 
In the case $t \in (h,1]$ we use Lemma \ref{lem:Pn-P} and the above estimates for $Q$ and $Q^n$ to write
\begin{equs}
\|O_t(x) -O^n_t(x)\|_{L_2(\Omega)} &  \leq  \Big(  \int_0^{t-h} \| p_{t-r}(x- \cdot)-p^{n}_{\kappa_n(t-r)}(x,\cdot) \|^2_{L_2(\T)} \, dr \Big) ^{1/2}
\\
&  \qquad +  \Big(  \int_{t-h}^t \| p_{t-r}(x - \cdot)-p^{n}_{\kappa_n(t-r)}(x,\cdot)  \|^2_{L_2(\T)} \, dr \Big) ^{1/2}
\\
& \lesssim  n^{-\beta/2} t^{(1-\beta)/4} + |Q(h)|^{1/2}+ |Q^n(h)|^{1/2}
\\
& \lesssim  n^{-\beta/2} t^{(1-\beta)/4}. 
\end{equs}
This finishes the proof.
\end{proof}
\begin{lemma}\label{lem:Q-difference}
For any $\beta\in[0,2)$ there exists a constant $N=N(c,\beta)$ such that for all $r\in[0,2]$ one has the bound
\begin{equ}\label{eq:Q-difference}
|Q^n(r)-Q(r)|\leq N n^{-\beta/2}r^{1/2-\beta/4}.
\end{equ}
\end{lemma}
\begin{proof}
Let us first assume $r\leq h$. Then by \eqref{eq:HK-L2} and \eqref{eq:Pn-initial-L2}
\begin{equ}
|Q^n(r)-Q(r)|\leq \int_0^r\int_\T|p^n_{\kappa_n(s)}(x,y)|^2+|p_s(x-y)|^2\,dy\,ds\lesssim r n+r^{1/2}.
\end{equ}
Since $r\leq h$ implies $rn\lesssim r^{1/2}$, the second term dominates. One similarly gets $r^{1/2}\lesssim r^{1/2-\beta/4}n^{-\beta/2}$ for any $\beta\geq 0$, yielding \eqref{eq:Q-difference}.
Moving on to the $r\geq h$ case, one has
\begin{equ}
|Q^n(r)-Q(r)|\leq I_1+I_2,
\end{equ}
with
\begin{equs}
I_1&=\int_0^h\int_\T|p^n_{\kappa_n(s)}(x,y)|^2+|p_s(x-y)|^2\,dy\,ds\lesssim n^{-1},
\\
I_2&=\int_h^r\Big(\int_\T|p^n_{\kappa_n(s)}(x,y)-p_s(x-y)|^2\,dx\Big)^{1/2}
\Big(\int_\T|p^n_{\kappa_n(s)}(x,y)+p_s(x-y)|^2\,dx\Big)^{1/2}\,ds
\\
&\lesssim \int_0^r n^{-\beta/2}s^{-(\beta+1)/4}s^{-1/4}\,ds,
\end{equs}
using \eqref{eq:Pn-P} to get the last line. As long as $\beta<2$, the last integral is finite and is of the required order. Finally, the condition $r\geq h$ implies $n^{-1}\lesssim n^{-\beta/2}r^{1/2-\beta/4}$, finishing the proof.
\end{proof}

For very short times, that is, $r\in[0,h]$, one has from \eqref{eq:Pn-initial-L2}
\begin{equ}\label{eq:Qn-small-times}
Q^n(r)=2nr.
\end{equ}
Otherwise, we have the following control on $Q^n$.
\begin{lemma}    \label{lem:bounds-Q-n}
For any $\beta \in [0,1]$, there exist constants $N>0$ depending only on $\beta$ and $c$  such that 
for any  $r\in[h ,1]$, $r' \in [ r, 1]$, we have 
\begin{equs}
Q^n(r)&\geq N^{-1}\sqrt{r},\label{eq:Q-lower-bound}
\\
|Q^n(r)-Q^n(r')| & \leq  N|r-r'|^\beta  |r|^{-\beta /2}|r'|^{(1-\beta)/2}.\label{eq:Q-upper-bound}
\end{equs}
\end{lemma}
\begin{proof}
First we show \eqref{eq:Q-lower-bound}.
Define
\begin{equ}
\tilde Q^n(r):=\int_{r/2}^{r}\int_\T |p^{n}_{\kappa_n (s)}(x,y)|^2\,dy\,ds,
\qquad
\tilde Q(r):=\int_{r/2}^{r}\int_\T |p_{s}(x-y)|^2\,dy\,ds.
\end{equ}
Since $\tilde Q^n\leq Q^n$, it suffices to bound $\tilde Q^n$.
From \eqref{eq:HK-L2} one gets for all $r\in(0,1]$
\begin{equ}\label{eq:Q-easy-bound}
N_1^{-1} \sqrt{r}\leq \tilde Q(r)\leq N_1\sqrt{r}
\end{equ}
with some absolute constant $N_1>0$.
By the triangle inequality one can write
\begin{equs}
|\big( & \tilde Q^n(r)\big)^{1/2}-\big(\tilde Q(r)\big)^{1/2}|
 \leq \Big(\int_{r/2}^r\int_\T|p_s(x-y)-p^{n}_{\kappa_n (s)}(x,y) |^2 \, dy\,ds\Big)^{1/2}\,.
\end{equs}
Note that if $N'$ is a sufficiently large constant, then $r\geq N'n^{-2}$ implies $r/2\geq h$.
Therefore the bound \eqref{eq:Pn-P} with the choice $\beta=1$ yields
\begin{equ}
\int_{r/2}^r\int_\T|p_s(x-y)-p^{n}_{\kappa_n (s)}(x,y) |^2 \, dy\,ds\leq N \int_{r/2}^r n^{-1} s^{-1}\,ds\leq N n^{-1}.
\end{equ}
Therefore, one has
\begin{equ}\label{eq:Qn-Q}
|\big(\tilde Q^n(r)\big)^{1/2}-\big(\tilde Q(r)\big)^{1/2}|\leq N n^{-1/2}.
\end{equ}
Combining \eqref{eq:Q-easy-bound} and \eqref{eq:Qn-Q} with the elementary equality $a^2\geq2b(a-b)+b^2$
\begin{equs}
\tilde Q^n(r)&
\geq 2\big(\tilde Q(r)\big)^{1/2}\Big(\big(\tilde Q^n(r)\big)^{1/2}-\big(\tilde Q(r)\big)^{1/2}\Big)+\tilde Q(r)
\\
&\geq -N\sqrt[4]{r}n^{-1/2}+N_1^{-1}\sqrt{r}.
\end{equs}
Choosing $N'$ sufficiently large, the bound \eqref{eq:Q-lower-bound} indeed follows for $r\geq N'n^{-2}$. If $r \in [h , N' n^{-2}$, then by the monotonicity of $Q^n$ and the fact that $Q^n(h)=2 n h $ (see \eqref{eq:Qn-small-times}),  we get 
\begin{equs}
Q^n(r) \geq Q^n(h) = 2n h = 2c n^{-1} \geq 2 c (\sqrt{N'})^{-1} \sqrt{r}, 
\end{equs}
which shows \eqref{eq:Q-lower-bound} also for $r \in [h , N' n^{-2}]$. 

We continue with \eqref{eq:Q-upper-bound}. By \eqref{eq:Pn-exponential} and Proposition \ref{prop:summation} we get
\begin{equs}
|Q^n(r)-Q^n(r')| &=  \int_{r}^{r'}  \sum_{j=-n}^{n-1} (1+h \lambda^n_j)^{  2 \kappa_n(t) h^{-1} }  \, dt \lesssim   \int_{r}^{r'}  \sum_{j\in\Z} e^{-2\delta j ^2 t}  \, dt
\\
&\lesssim \int_{r}^{r'}  t^{-1/2} \, dt 
\leq |r-r'| |r|^{-1/2}. \label{eq:sqrt-blow-up}
\end{equs}
On the other hand, we can estimate the difference $|Q^n(r)-Q^n(r')|$ term by term. Recalling \eqref{eq:Q-difference} and \eqref{eq:Q-easy-bound-first}, we get
$
|Q^n(r'')| \leq |Q^n(r'')-Q(r'')|+|Q(r'')|\lesssim (r'')^{1/2}
$
for both $r''=r,r'$.
It follows that
$|Q^n(r)-Q^n(r')| \lesssim (r')^{1/2}$, which combined with \eqref{eq:sqrt-blow-up} 
 finishes the proof.
\end{proof}

\begin{lemma}    \label{lem:regularity-dis-OU}
For any $p>0$, $\theta \in (0, 1/2)$, there exist a constant $N(c,p,\theta)$ such that
\begin{equs}
\sup_{n \in \mathbb{N}}\sup_{t  \leq 1} \| O^n_t \|_{L_p(\Omega; \cC^{1/2-\theta}(\T))} \leq N.
\end{equs}
\end{lemma}

\begin{proof}
For $t \in [0,1]$, $x,z \in \T$,  by the Burkholder-Davis-Gundy inequality we have 
\begin{equs}
\| O^n_t(x)-O^n_t(z)\|_{L_p(\Omega)}^2 \lesssim \int_0^t  \| p^n_{\kappa_n(t-s)} (x, \cdot)- p^n_{\kappa_n(t-s)} (z, \cdot) \|_{L_2(\T)}^2 \, ds 
\end{equs}
By \eqref{eq:Pn-exponential} and the bounds $\| e_n\|_{\mathbb{B}(\T)} \lesssim 1$,  $\| e_n\|_{\C^1(\T)} \lesssim n$, we have 
\begin{equs}
\| p^n_{\kappa_n(t-s)} (x, \cdot)- p^n_{\kappa_n(t-s)} (z, \cdot) \|_{L_2(\T)}^2 & = \sum_{j=-n}^{n-1} |1+h \lambda^n_j|^{2\kappa_n(t-s)  h^{-1}} |e^n_j(x)- e^n_j(z)|^2 
\\
&\leq \sum_{j=-n}^{n-1}  e^{-2\kappa_n(t-s)\delta j^2}|x-z|^{(1-2\theta)} n^{(1-2\theta)}.
\end{equs}
Consequently,  by Proposition \ref{prop:summation} we get 
\begin{equs}
 \| O^n_t(x)  -O^n_t(z)\|_{L_p(\Omega)}^2
& \lesssim  \int_0^{(t-2h)\vee 0}\sum_{j=-n}^{n-1}  e^{-2\kappa_n(t-s)\delta j^2}|x-z|^{(1-2\theta)} n^{(1-2\theta)} \, ds 
\\
& \qquad + \int_{(t-2h)\vee 0}^t \sum_{j=-n}^{n-1}  |x-z|^{(1-2\theta)} n^{(1-2\theta)} \, ds
\\
& \lesssim  |x-z|^{1-2\theta} \int_0^{(t-2h)\vee 0}   |t-s|^{-1+2\theta} \, ds +N(p)  |x-z|^{(1-2\theta)}  n^{-2\theta} 
\\
& \lesssim   |x-z|^{1-2\theta}.
\end{equs}
Similarly, one sees that 
\begin{equs}
\| O^n_t(0)\|_{L_p(\Omega)}^2  \lesssim 1,
\end{equs}
which combined with the above estimate gives 
\begin{equs}
\| O^n_t\|_{\cC^{1/2-\theta}(\T;  L_p(\Omega))}  \lesssim 1.
\end{equs}
Since $p $ and $\theta$ are arbitrary, by Kolmogorov's continuity criterion  we get 
\begin{equs}
\| O^n_t\|_{L_p(\Omega; \cC^{1/2-\theta}(\T))}  \lesssim 1,
\end{equs}
and the claim follows since the bound does not depend on $t$ or $n$. 
\end{proof}

\begin{lemma}  \label{lem:uniform-C-1/2}
Under the assumption of Theorem \ref{thm:main-theorem} there exists a constant $N(c,\eps, K, p, \| b\|_{\mathbb{B}})$ such that 
\begin{equs}
\sup_{n \in \mathbb{N}} \sup_{t \in [0,1]} \| u^n_t \|_{L_p(\Omega; \cC^{1/2-\eps}(\T))}  \leq N. 
\end{equs}
\end{lemma}

\begin{proof}
Let us set denote  $v^n_t = \cP^n_t \psi^n+O^n_t$. The conclusion of the lemma with $u^n$ replaced by $v^n$  is an immediate consequence of Lemma \ref{lem:preservation-reg-T} and Lemma \ref{lem:regularity-dis-OU}. 

From Girsanov's theorem (see e.g. \cite[Thm~10.14]{DPZ} for a sufficiently general version) one has that
under the measure $\tilde\bP$ defined by
\begin{equ}\label{eq:rad-nyk}
\frac{d\tilde\bP}{d\bP}=\rho=\exp\Big(-\int_0^1\int_\T b\big(u^n_t(x)\big)\,\xi(dx, dt)-\frac{1}{2}\int_0^1\int_\T \big|b\big(u^n_t(x)\big)\big|^2\,dx\,dt\Big),
\end{equ}
the mapping
\begin{equ}
A \mapsto\xi(A)+\scal{b(u^n),\bone_A}_{L_2([0,1]\times \T)}
\end{equ}
from $\cB([0,1]\times \T)$ to $L_2(\Omega)$
defines a white noise. In particular, the law of $u^n$ under $\tilde\bP$ and the law of $v^n$ under $\bP$ coincide.
It is also an easy exercise that 
$\E\rho^{-1}\leq N(\|b\|_{\mathbb{B}})< \infty$. 
Therefore,
\begin{equs}
\| u^n_t \|^p_{L_p(\Omega; \cC^{1/2-\eps}(\T))} =\tilde\E\big(\| u^n_t \|^p_ {\cC^{1/2-\eps}(\T))}\rho^{-1}\big)&\leq \big(\tilde \E \| u^n_t \|^{2p}_ {\cC^{1/2-\eps}(\T))}\big)^{1/2}\big(\tilde\E\rho^{-2}\big)^{1/2}
\\
&=\big(\E\| v^n_t \|^{2p}_ {\cC^{1/2-\eps}(\T))}\big)^{1/2}\big(\E\rho^{-1}\big)^{1/2}
\leq N.
\end{equs}
This finishes the proof.  
\end{proof}

\section{The sewing strategy}
As already discussed,  one of the main tools for proving our main results is the stochastic sewing lemma. First we give an outline of the strategy, identify the various terms to be bounded, and then carry out the estimates.
\subsection{Overview}\label{sec:outline}

Here we give a brief overview of the strategy of the proof. For reference, we will compare to the $1$-dimensional additive SDE
\begin{equ}
dX_t=f(X_t)\,dt+dW_t
\end{equ}
driven by a standard Wiener process $W$. Let us assume $f\in \cC^\alpha(\R)$ with some $\alpha\in(0,1)$. The Euler-Maruyama approximation of the SDE reads as
\begin{equ}
dX^n_t=f(X^n_{\hat\kappa_n(t)})\,dt+dW_t,
\end{equ}
where we briefly use the notation $\hat\kappa_n(t)=\lfloor nt\rfloor n^{-1}$. Assuming identical initial conditions, one can decompose the error as
\begin{equ}\label{eq:simple-error}
X_t-X^n_t=\int_0^tf(X_s)-f(X^n_s)\,ds+\int_0^tf(X^n_s)-f(X^n_{\hat\kappa_n(s)})\,ds.
\end{equ}
One then aims to bound the first term by $|X-X^n|$ with \emph{some} norm $|\cdot|$ and the second by a negative power of $n$, which can in fact be $n^{-(1+\alpha)/2}$. If one furthermore achieves a small constant (say, less than $1/2$) in the first bound, then the inequality buckles and the error itself is bounded by $n^{-(1+\alpha)/2}$.

Of course neither of these tasks are really obvious, since simply bounding the integrals by bringing the absolute value inside gives the bounds $t\|X-X^n\|_{L_\infty([0,t])}^\alpha$ and $n^{-\alpha}$, respectively. The former is particularly problematic, since buckling arguments (or equivalently, Gronwall-type lemmas) fail for powers strictly less than $1$. 
In \cite{BDG} this issue is overcome by stochastic sewing approach, which however requires to work with a stronger norm: the choice $|\cdot|=\|\cdot\|_{\cC^{1/2}([0, 1]; L_p(\Omega))}$ suffices for example. On one hand, this has the advantage of providing the final error estimates in a strong norm, the drawback is that instead of \eqref{eq:simple-error} one has to control the increments of the error as well.

In infinite dimensions there are several issues with this strategy.  We have 
\begin{equs}
u_t-u^n_t= \cP_t\psi-\cP^n_t\psi^n & +  \int_0^t p_{t-r} *\big( b (u_r)\big) \, dr - \int_0^t p^n_{\kappa_n(t-r) }*_n\big( b(u^n_{\kappa_n(r)}) \big)\, dr
+O_t-O^n_t.
\end{equs}
First, the quantity $u_t-u_s$ does not have a natural form as an integral from $s$ to $t$.
Second, even if one considers the ``mild'' increments $u_t-\cP_{t-s}u_s$, there is no nice analogous increment for the approximate solution.
Instead, we study the quantity  
\begin{equ}\label{eq:main-error-guy}
\cE_{s,t}^n=\int_s^t p_{t-r} *\big( b (u_r)\big) \, dr - \int_s^t p^n_{\kappa_n(t-r) }*_n\big( b(u^n_{\kappa_n(r)}) \big)\, dr.
\end{equ}
The above is not an increment (not even mild), however, it is 
an analogue of the  increments of the  right-hand side  of \eqref{eq:simple-error},  in the infinite dimensional case,  which serves its purpose.

We will use the decomposition
\begin{equs}[eq:main-error-decomposition]
\cE_{s,t}^n=\cE_{s,t}^{n,1}+\cE_{s,t}^{n,2}+\cE_{s,t}^{n,3}
&:=\int_s^t \Big( p_{t-r} *  \big( b(u_r)-b(u^n_r) \big) \Big) \,dr  \\
&\qquad+ \int_s^t \Big(  p_{t-r} *\big(b (u^n_r)\big) - p_{t-r} *_n \big(b (u^n_{\kappa_n(r)} )  \big)\Big) \, dr 
\\
&\qquad+ \int_s^t \Big( \big(p_{t-r}  -  p^n_{\kappa_n(t-r)}\big) *_n\big(b (u^n_{\kappa_n(r)}) \big)\Big)  \, dr.
\end{equs}
Our goal will be to estimate the term $\cE^{n,1}$ in terms of $\cE^n$, which will lead to buckling for $\cE^n$, and the remaining $\cE^{n,2}, \cE^{n,3}$ by some power of $n$. Both of these steps will be achieved by  the stochastic sewing lemma. Finally, notice that the above procedure  will give an estimate for $\cE^n$ and not $u-u^n$ itself. The reason that we follow this route  will become clearer later, see Remark \ref{rem:buckling-remark}.

We now recall the stochastic sewing lemma. The notation $[a, b]_{<}$ below  stands for $\{ (s, t) \in [a, b]^2 : s<t\}$. 
\begin{theorem} [{\cite[Theorem 2.4]{Khoa}}]           \label{thm:SSL}
Let $p\geq 2$, $0\leq s'\leq t'\leq 1$ and let $A_{\cdot,\cdot}$ be a function $[s',t']_<\to L_p(\Omega)$ such that for any $(s,t)\in[s',t']_<$ the random variable $A_{s,t}$ is $\F_t$-measurable. Suppose that for some $\eps_1,\eps_2>0$ and $C_1,C_2$ the bounds
\begin{equs}
\|A_{s,t}\|_{L_p(\Omega)} & \leq C_1|t-s|^{1/2+\eps_1},\label{eq:SSL-cond1}
\\
\|\E^s\delta A_{s,u,t}\|_{L_p(\Omega)} & \leq C_2 |t-s|^{1+\eps_2}\label{eq:SSL-cond2}
\end{equs}
hold for all $s' \leq s \leq u \leq t \leq t'$, where $\delta A_{s,u,t}:=A_{s, t}-A_{s, u}-A_{u, t}$. 
Then there exists a unique map  $\A:[s',t']\to L_p(\Omega)$ 
with the following three properties: 
\begin{enumerate}[(i)]
\item With probability one,  $\A_{s'}=0$. 

\item $ \A_t$ is $\cF_t$-measurable for all $t \in [s', t']$,

\item   There exists constants  $K_1,K_2>0$ such that for all $(s, t) \in [s', t']_<$ we have 
\begin{equs}
\|\A_t	-\A_s-A_{s,t}\|_{L_p(\Omega)} & \leq K_1 |t-s|^{1/2+\eps_1},         \label{eq:SSL-conc1}
\\
\|\E^s\big(\A_t	-\A_s-A_{s,t}\big)\|_{L_p(\Omega)} & \leq K_2|t-s|^{1+\eps_2}.  \label{eq:SSL-conc2}
\end{equs}
\end{enumerate}
In addition,  there exists a constant $K>0$ depending only on $\eps_1,\eps_2$ and $p$,  such that for all $(s, t) \in [s', t']_<$ we have 
\begin{equation}\label{eq:SSL-conc3}
\|\A_t-\A_s\|_{L_p(\Omega)}  \leq  KC_1 |t-s|^{1/2+\eps_1}+KC_2 |t-s|^{1+\eps_2}.
\end{equation}
\end{theorem}
\begin{remark}
It can be sometimes convenient to incorporate the semigroup of the (linear part of the) equation into the the formulation of the stochastic sewing lemma, as in \cite{mild-sewing}.
This is indeed how the first version of the present paper proceeded, but as noted by a referee, it is easier to reduce the argument to the original stochastic sewing lemma, similarly to \cite{ABLM}.
\end{remark}

\subsection{Estimate for $\cE^{n,1}$}

Let us  introduce the following (semi)norms.
Let $(s',t')\in[0,1]_<$ and $p\in[2,\infty]$. For a map $\varphi : [s',t'] \times \T \to L_p (\Omega)$  we set
\begin{equ}
\|\varphi \|_{\scC^0_p[s',t']}=\sup_{x\in\T}\sup_{s\in[s',t']}\| \varphi _s(x)\|_{L_p(\Omega)}.
\end{equ}
Furthermore, for  $\alpha\in(0,1]$ and a map $\varphi : [s',t']_<  \times \mathbb{T} \to L_p (\Omega)$ we set
\begin{equs}
\,[\varphi]_{\scC^\alpha_p[s',t']_<}=\sup_{x\in\T}\sup_{(s,t)\in[s',t']_{<}}\frac{\|\varphi_{s,t}(x)\|_{L_p(\Omega)}}{|t-s|^\alpha}.
\end{equs}
Although our goal is to bound $\cE^{n,1}_{s,t}$, it is useful to introduce the generalised quantity
\begin{equ}
\cE^{n,1}_{s,t}[f]=\int_{s}^t \cP_{t-r}\big(  f(u_r)  - f(u^n_r)  \big)\, dr,  \qquad f \in \mathbb{B}(\R) .
\end{equ}

\begin{lemma}    \label{lem:regularisation-lemma}
Let $\tau\in(1/4,3/4)$   and $(s',t')\in[0,1]_<$. 
Then, under the assumption of Theorem \ref{thm:main-theorem},  for all $f \in \mathbb{B}(\R)$, $n\in\N$, and  $(s,t)\in[s',t']_<$ the following bound holds
\begin{equs}
\sup_{x\in\T} \| \cE^{n,1}_{s,t}[f] (x)\|_{L_p(\Omega)}
\leq& N \| f\|_{\mathbb{B}(\R)} |t-s|^{3/4} \big( n^{-1/2+ \eps}+\sup_{x \in \T}\| \psi (x)  -\psi^n (x)\|_{L_p(\Omega)}+\| \mathcal{E}^n_{0, \cdot}  \|_{\mathscr{C}^0_p[s', t']} \big)    
\\
& +  N \| f\|_{\mathbb{B}(\R)} |t-s|^{3/4+\tau}    [ \mathcal{E}^n ]_{\mathscr{C}^\tau_p[s', t']_<},  \label{eq:reg bound}
\end{equs}
where the constant $N$ depends only on  $\|b\|_{\mathbb{B}(\R)},c, \eps, p, K$,  and $\tau$.
\end{lemma}

\begin{remark}  \label{rem:buckling-remark}
Notice that  the right-hand side contains the term $ [ \mathcal{E}^n ]_{\mathscr{C}^\tau_p[s', t']_<}$,  where $\tau>1/4$. This is the reason that we aim to buckle for $\cE^n$ and not $u-u^n$ itself, as  the latter has no more than $1/4 $ regularity in time, because of the term $O-O^n$. 
\end{remark}

\begin{proof}

By linearity in $f$, we may and will assume $\|f\|_{\mathbb{B}(\R)}=1$. 
We first assume that $f$ is in addition Lipschitz, derive the bound \eqref{eq:reg bound} that does not depend on its Lipschitz norm, and then conclude with a standard approximation argument.
We fix $x \in \mathbb{T}$,  and for $(s,t)\in[s',t']_<$ we define 
\begin{equs}
A_{s, t} (x) :=\E^s \int_s^t \Big( \cP_{t'-r}  \big( f  (O_r+ \phi_{s, r} \big) ) -\cP_{t'-r} \big( f  (O^n_r+\phi^n_{s,r}) \big )  \Big) (x)\,dr,
\end{equs}
where 
\begin{equs}
\phi_{s, r}(y)  & = \cP_{r}  \psi (y) + \E^s \int_0^r  \cP_{r-\theta}  \big( b ( u_\theta)  \big)(y)  \, d \theta, 
\\
\phi^n_{s, r}(y)  & = \cP^n _{r}  \psi^n(y) + \E^s \int_0^r \cP_{\kappa_n(r-\theta)} ^n \big(  b  (u^n_{\kappa_n(\theta)})  \big) (y) \, d \theta.
\end{equs}
We aim to verify the conditions of Theorem \ref{thm:SSL}. We start by \eqref{eq:SSL-cond1}, that is, by obtaining an estimate for $\|A_{s,t}(x)\|_{L_p(\Omega)}$.
First of all, notice that we can interchange the action of $\E^s$ and $\cP_{t'-r}$, and therefore by
 \eqref{eq:OU-transition-main} and \eqref{eq:OUn-transition-main} one can write
\begin{equ}
A_{s,t}(x)=\int_s^t\cP_{t'-r}B_r(x)\,dr,
\end{equ}
with
\begin{equs}
B_r(y) : = \cP^\R_{Q(r-s)}f \big( \cP_{r-s}O_s(y)+\phi_{s, r} (y)\big) -  \cP^\R_{Q^n(r-s)}f \big( \widehat{O}^n_{s,r}(y)+\phi^n_{s, r}(y) \big).
\end{equs}
First consider the case $t\geq s+h$. We then have
\begin{equ}
A_{s,t}(x)=I_1(x)+I_2(x):=\Big(\int_s^{s+h}+\int_{s+h}^t\Big)\cP_{t-r}B_r(x)\,dr.
\end{equ}
For $I_1$ we have the trivial estimate
\begin{equ}\label{eq:regu-I1}
\|I_1(x)\|_{L_p(\Omega)}\lesssim   h  \lesssim  n^{-1/2} (t-s)^{3/4}.
\end{equ}
As for $I_2$,  by applying \eqref{eq:HK bound mixed}  we get
\begin{equs}
|B_r(y) | &
 \lesssim  \big(|  \cP_{r-s}O_s(y)+\phi_{s, r} (y)-  \widehat{O}^n_{s,r}(y)-\phi^n_{s, r}(y) )| 
 +| Q(r-s)-Q^n(r-s)|^{1/2}\big)
 \\
 &\qquad\times \big( Q(r-s) \wedge Q^n(r-s) \big)^{-1/2}                    \label{eq:x05}
\end{equs}
Next, using \eqref{eq:Estimate-BDG} with $\beta=1-2\eps$ gives,
\begin{equs}
\|\cP_{r-s}O_s(y)-\widehat{O}^n_{s,r}(y)\|_{L_p(\Omega)}&\lesssim n^{-1/2+\eps}|r-s|^{\eps/2} \lesssim n^{-1/2+\eps} . \label{eq:x01}
\end{equs}
By using 
 \eqref{eq:rate-deterministic} with $\alpha=1/2-\eps$,   and the assumption of the theorem, we get 
\begin{equs}
&\| \phi_{s, r}(y) - \phi^n_{s, r} (y) \|_{L_p(\Omega)} 
\\
 \leq &\| \cP^n_r \psi^n(y) - \cP_r \psi (y)\|_{L_p(\Omega)}+ \| \mathcal{E}^n_{0, \cdot}  \|_{\mathscr{C}^0_p[s', t']}
\\
 \leq & \| \cP^n_r \psi^n(y) - \cP_r \psi ^n(y) \|_{L_p(\Omega)}+  \| \cP_r \psi^n(y) - \cP_r \psi (y) \|_{L_p(\Omega)}+ \| \mathcal{E}^n_{0, \cdot}  \|_{\mathscr{C}^0_p[s', t']} 
\\
 \lesssim  & n^{-1/2+ \eps} \| \psi^n\|_{L_p(\Omega; C^{1/2-\eps}(\T))} +  \big(\cP_r \| \psi ( \cdot) -\psi^n(\cdot) \|_{L_p(\Omega)} \big) (y)  + \| \mathcal{E}^n_{0, \cdot}  \|_{\mathscr{C}^0_p[s', t']} 
 \\
 \lesssim  &  n^{-1/2+ \eps}  + \sup_{y \in \T}\| \psi (y)  -\psi^n (y)\|_{L_p(\Omega)}  + \| \mathcal{E}^n_{0, \cdot}  \|_{\mathscr{C}^0_p[s', t']} 
  \label{eq:x02}
\end{equs}
Moreover, by using  \eqref{eq:Q-difference} with $\beta=2-4\eps$ we get 
\begin{equs}
|Q(r-s)-Q^n(r-s)|^{1/2}&\lesssim n^{-1/2+\eps}|r-s|^{\eps/2} \lesssim n^{-1/2+\eps}.\label{eq:x03}
\end{equs}
We now combine \eqref{eq:x05} with \eqref{eq:x01}-\eqref{eq:x03}, and by keeping in mind  that $Q(r-s)\gtrsim |r-s|^{1/2}$  for all $r,s\in[s',t']_<$, and that by \eqref{eq:Q-lower-bound} we also have  $Q^n(r-s)\gtrsim |r-s|^{1/2}$  for $r \geq s+h$, we conclude that 
\begin{equs}
\sup_{y \in \T} \| B_r(y) \|_{L_p(\Omega)} \lesssim  |r-s|^{-1/4} \big( n^{-1/2+\eps}+\sup_{y \in \T}\| \psi (y)  -\psi^n (y)\|_{L_p(\Omega)}+\| \mathcal{E}^n_{0, \cdot}  \|_{\mathscr{C}^0_p[s', t']} \big).
\end{equs}
This in turn implies that 
\begin{equs}
\|I_2(x)\|_{L_p(\Omega)}\lesssim
 |t-s|^{3/4} \big( n^{-1/2+\eps}+\sup_{y \in \T}\| \psi (y)  -\psi^n (y)\|_{L_p(\Omega)}+\| \mathcal{E}^n_{0, \cdot}  \|_{\mathscr{C}^0_p[s', t']} \big).\label{eq:regu-I2}
\end{equs}
Hence, in the regime $t\geq s+h$ we have from \eqref{eq:regu-I1} and \eqref{eq:regu-I2} that
\begin{equs}         
\|  A_{s,t}(x)\|_{L_p(\Omega)}  \lesssim
|t-s|^{3/4} \big( n^{-1/2+\eps}+\sup_{y \in \T}\| \psi (y)  -\psi^n (y)\|_{L_p(\Omega)}+\| \mathcal{E}^n_{0, \cdot}  \|_{\mathscr{C}^0_p[s', t']} \big).
\\
 \label{eq:estimate-A-1/2}
\end{equs}
If $t \in [s, s+h)$, we can simply use a trivial bound: 
\begin{equs}
\|  A_{s,t}(x)\|_{L_p(\Omega)}   \leq 2 |t-s| \lesssim n^{-1/2}  |t-s|^{3/4}. 
\end{equs}
We conclude that
\eqref{eq:SSL-cond1} is satisfied with 
$$
C_1= N \big( n^{-1/2+\eps}+\sup_{y \in \T}\| \psi (y)  -\psi^n (y)\|_{L_p(\Omega)}+\| \mathcal{E}^n_{0, \cdot}  \|_{\mathscr{C}^0_p[s', t']} \big)
$$
and $\eps_1=1/4$. 

Next,
let us bound the term 
$\|\E^s \delta  A_{s,u,t}(x)\|_{L_p(\Omega)}$.  A simple calculation shows that 
\begin{equs}
\E^s \delta A_{s,u,t}(x) =& \E^s \E^u \int_u^t  \Big(  \cP_{t'-r}  \big( f  ( O_r+ \phi_{s, r}) \big) -\cP_{t'-r}  \big( f  ( O^n_r+ \phi^n_{s, r}) \big)  \Big) (x) \, dr 
\\
&- \E^s \E^u \int_u^t \Big(  \cP_{t'-r} \big(  f  ( O_r+ \phi_{u, r}) \big) -\cP_{t'-r}  \big( f ( O^n_r+ \phi^n_{u, r}) \big) \Big)(x)\, dr . 
\end{equs}
Similarly to before, we write
\begin{equs}  
\E^s \delta  A_{s,u,t}(x)   = \E^s  \int_u^t \cP_{t'-r}  D_r (x)\, dr ,
 \label{eq:delta-integrand}
\end{equs}
with
\begin{equs}
D_r(y) &:= \cP^\R_{Q(r-u)}  f ( \cP_{r-u}O_u (y) + \phi_{s, r} (y)    \big) -\cP^\R_{Q^n(r-u)}  f\big( \widehat{O}^n_{u,r}(y) + \phi^n_{s, r}(y)   \big)
\\
&\qquad- \cP^\R_{Q(r-u)}  f \big( \cP_{r-u}O_u(y) + \phi_{u, r}(y) \big) +\cP^\R_{Q^n(r-u)}   f \big( \widehat{O}^n_{u,r}(y)+ \phi^n_{u, r}(y)\big) .
\end{equs}
Let us start by a rough estimate when $|r-u|\leq h$.
We pair up the first and third, and the second and fourth terms in $D_r(y)$ and apply \eqref{eq:HK bound mixed} (with $\beta=1$ and $\alpha=0$).  This combined with      \eqref{eq:Q-easy-bound-first} and  \eqref{eq:Qn-small-times} gives
\begin{equs}
|D_r(y)| & \lesssim   |Q(r-u)|^{-1/2} | \phi_{s, r}(y)-\phi_{u, r}(y)|+ |Q^n(r-u)|^{-1/2} | \phi^n_{s, r}(y)-\phi^n_{u, r}(y)| 
\\
& \lesssim  |r-u|^{-1/4}  | \phi_{s, r}(y)-\phi_{u, r}(y)|+ n^{-1/2} |r-u|^{-1/2}| \phi^n_{s, r}(y)-\phi^n_{u, r}(y)|. 
\end{equs}
Notice that  for all  $X, Y \in L_\infty(\Omega)$ with $Y$ being $\cF_s$-measurable,  by the triangle inequality, conditional Jensen's inequality, and the monotonicity of the conditional expectation, we have $\| \E^s X-X \|_{L_\infty(\Omega)} \leq 2 \| Y-X\|_{L_\infty(\Omega)}$.  By using this, we see that 
\begin{equs}
\|\phi_{s, r}(y)- \phi_{u, r} (y)\|_{L_\infty(\Omega)} & =  \Big\| \E^s\E^u \int_0^r \cP_{r-\theta} \big( b ( u_\theta)\big) (y)  \, d \theta- \E^u \int_0^r \cP_{r-\theta}  \big( b ( u_\theta)  \big)  (y)\, d \theta \Big\|_{L_\infty(\Omega)}
\\
&\leq     2\Big\| \E^u  \int_0^s \cP_{r-\theta} \big( b ( u_\theta) \big) (y)  \, d \theta- \E^u \int_0^r  \cP_{r-\theta}\big( b  (u_\theta ) \big)(y)  \, d \theta \Big\|_{L_\infty(\Omega)}
\\
&\lesssim    |r-s|.  \label{eq:phi-t-s}
\end{equs}
Similarly, we get 
\begin{equs}       \label{eq:phi-n-t-s}
\|\phi^n_{s, r}(y)- \phi^n_{u, r}(y) \|_{L_\infty(\Omega)}  \lesssim |r-s|.
\end{equs}
Therefore
\begin{equs}\label{eq:D-easy}
\sup_{y \in \T} \|D_r(y) \|_{L_p(\Omega)} \lesssim \Big(  (r-u)^{-1/4} (r-s) + n^{-1/2} (r-u)^{-1/2}(r-s)\Big).
\end{equs}
Let us now first deal with the case $t \in [ u, u+h)$.
Putting the above bound into \eqref{eq:delta-integrand} we get
\begin{equs}
 \|\E^s \delta A_{s,u,t}(x)\|_{L_p(\Omega)}   &  \leq \int_u^t \sup_{y  \in {\T}}\|D_r(y) \|_{L_p(\Omega)} \, dr  
\\
 & \lesssim  \Big( (t-u)^{3/4} (t-s)+ n^{-1/2} (t-s)^{3/2} \Big) 
 \\
&  \lesssim n^{-1/2} (t-s)^{3/2}.\label{eq:regu-delta-easy}
\end{equs}
Moving on to the case $t\geq u+h$, we write
\begin{equs}
\|\E^s \delta  A_{s,u,t}(x)\|_{L_p(\Omega)}   & \leq
\|I_1\|_{L_p(\Omega)}+\|I_2\|_{L_p(\Omega)}:=\Big(\int_u^{u+h}+\int_{u+h}^t\Big)\sup_{y \in {\T}}  \|D_r(y) \|_{L_p(\Omega)} \, dr. 
\end{equs}
For $I_1$ we may use \eqref{eq:D-easy} again to get
\begin{equ}
\|I_1\|_{L_p(\Omega)}\lesssim n^{-3/2}|t-s|\lesssim n^{-1/2}|t-s|^{3/2}.
\end{equ}
As for $I_2$, we decompose the integrand as $D_r=D^1_r+D^2_r$, where 
\begin{equs}
D^1_r(y):&= \cP^\R_{Q(r-u)}   f \big( \widehat{O}^n_{u,r}(y)+ \phi^n_{s, r}(y)\big)  -\cP^\R_{Q(r-u)}   f\big( \widehat{O}^n_{u,r}(y) + \phi^n_{u, r}(y) \big) 
\\
&\qquad-\cP^\R_{Q^n(r-u)}  f \big( \widehat{O}^n_{u,r}(y) + \phi^n_{s, r}(y)  \big) +\cP^\R_{Q^n(r-u)}  f \big(\widehat{O}^n_{u,r}(y) + \phi^n_{u, r} (y) \big);
\\
D^2_r(y)  :&= \cP^\R_{Q(r-u)} f  \big(  \cP_{r-u}O_u(y)+ \phi_{s, r}(y)\big)  -\cP^\R_{Q(r-u)}  f \big(  \widehat{O}^n_{u,r}(y) + \phi^n_{s, r}(y) \big) 
\\
&\qquad- \cP^\R_{Q(r-u)}   f \big(  \cP_{r-u}O_u(y) + \phi_{u, r}(y)  \big) +\cP^\R_{Q(r-u)}  f \big(  \widehat{O}^n_{u,r}(y) + \phi^n_{u, r}(y)\big)  
\end{equs}
For $D^1_r$ we use \eqref{eq:HK bound 4} to obtain
\begin{equs}
|D^1_r(y) | &\lesssim \big(  Q(r-u) \wedge Q^n(r-u) \big) ^{-3/2} |Q(r-u)-Q^n(r-u) | |\phi^n_{s, r}(y) - \phi^n_{u, r}(y) | 
\\
&\lesssim |r-u|^{-1/2}  n ^{-1/2}|\phi^n_{s, r}(y) - \phi^n_{u, r}(y) | ,
\end{equs}
where for the second inequality we have used  \eqref{eq:Q-easy-bound-first}, \eqref{eq:Q-lower-bound},  and \eqref{eq:Q-difference} (the latter with $\beta=1$). Consequently, by \eqref{eq:phi-n-t-s}, we get 
\begin{equs}   \label{eq:estimate-A-1}
\sup_{y \in \T}\|D^1_r(y) \|_{L_p(\Omega)} \lesssim |r-u|^{-1/2}  n ^{-1/2}|t-s|.
\end{equs}
For $D^2_r$ we use \eqref{eq:HK bound 3}  with $\alpha=0$, to get 
\begin{equs}
|D^2_r(y) | &\lesssim |Q(r-u)|^{-1}   |\cP_{r-u}O_u(y) + \phi_{s, r} (y) -\widehat{O}^n_{u,r}(y) - \phi^n_{s, r} (y)  ||\phi_{s, r}(y)  - \phi_{u, r}(y) | 
\\
&\qquad+ |Q(r-u)|^{-1/2}   |  \phi_{s, r}(y)  - \phi^n_{s, r}(y) -   \phi_{u, r} (y) +  \phi^n_{u, r}(y) |  .        \label{eq:A2-dec}
\end{equs}     
From \eqref{eq:x01}, \eqref{eq:x02}  and \eqref{eq:phi-t-s} we get that
\begin{equs}
 \Big\| &|\cP_{r-u}O_u(y) + \phi_{s, r} (y) -\widehat{O}^n_{u,r}(y) - \phi^n_{s, r}(y)   ||\phi_{s, r}(y)  - \phi_{u, r}(y) |  \Big\|_{L_p(\Omega)}  
 \\
&\qquad \lesssim \big( n^{-1/2+\eps} + \sup_{x \in \T}\| \psi (x)  -\psi^n (x)\|_{L_p(\Omega)}+ \| \mathcal{E}^n_{0, \cdot}  \|_{\mathscr{C}^0_p[s', t']} \| \big) |r-s|.      \label{eq:part-1-A2}
\end{equs}
Moreover, similarly to the argument for \eqref{eq:phi-t-s}, we get 
\begin{equs}
 \|  &\phi_{s, r}(y)  - \phi^n_{s, r}(y) -   \phi_{u, r} (y) +  \phi^n_{u, r}(y)  \| _{L_p(\Omega)} 
\\
&\leq 2 \Big\|  \E^u  \int_s^r \cP_{r-\theta} \big( b ( u_\theta)   \big)  (y) \, d \theta- \E^u \int_s^r \cP^n_{\kappa_n(r-\theta)}\big(  b  (u^n_{\kappa_n(\theta)})  \big)(y) \, d \theta \Big\|_{L_p(\Omega)}
\\
&\leq 2 \| \mathcal{E}^n_{s,r} (y) \|_{L_p(\Omega)}
\end{equs}
Therefore,  by the above estimates and \eqref{eq:Q-easy-bound-first}, we get that
\begin{equs}\label{eq:x04}
& \sup_{y\in\T}\|D_r^2(y)\|_{L_p(\Omega)}
\\
\lesssim & \big(n^{-1/2+\eps} +\sup_{y \in \T}\| \psi (y)  -\psi^n (y)\|_{L_p(\Omega)}+ \| \mathcal{E}^n_{0, \cdot}  \|_{\mathscr{C}^0_p[s', t']} +[ \mathcal{E}^n ]_{\mathscr{C}^\tau_p[s', t']_<}\big)|r-s|^{-1/4+\tau},
\end{equs}
where we used that $\tau<3/4$. Integrating the bounds \eqref{eq:estimate-A-1} and \eqref{eq:x04} with respect to $r$,
%
we conclude that 
\begin{equs}
&  \|\E^s \delta  A_{s,u,t}(x)\|_{L_p(\Omega)} 
\\ \lesssim &     \Big( n^{-1/2+\eps} +\sup_{y \in \T}\| \psi (y)  -\psi^n (y)\|_{L_p(\Omega)} + \| \mathcal{E}^n_{0, \cdot}  \|_{\mathscr{C}^0_p[s', t']} +  [ \mathcal{E}^n ]_{\mathscr{C}^\tau_p[s', t']_<}\Big) (t-s)^{3/4+\tau}  .
\end{equs}
This shows that \eqref{eq:SSL-cond2} is satisfied with 
$$
C_2= N   \Big( n^{-1/2+\eps} + \sup_{y \in \T}\| \psi (y)  -\psi^n (y)\|_{L_p(\Omega)}+ \| \mathcal{E}^n_{0, \cdot}  \|_{\mathscr{C}^0_p[s', t']} +  [ \mathcal{E}^n ]_{\mathscr{C}^\tau_p[s', t']_<}\Big) 
$$
 and $\eps_2= 3/4+\tau -1 >0$, where we used that $\tau>1/4$. Therefore Theorem \ref{thm:SSL} applies.

We claim that the map $\cA : [s',t'] \to L_p(\Omega)$  constructed in Theorem \ref{thm:SSL} coincides with 
\begin{equs}
\cA_t(x) := \int_{s'}^t \Big( \cP_{t'-r}  \big( f  (u_r)\big)  -\cP_{t'-r} \big( f  (u^n_r) \big )  \Big) (x)\,dr, \qquad t \in [s', t'].
\end{equs}
 First of all, it is obvious that  $\cA_t(x)$  is $\cF_t$-measurable for each $t \in [s',t']$ and that $\cA_{s'}=0$.
 Hence, we only have to check that $\cA_\cdot(x)$ satisfies \eqref{eq:SSL-conc1}-\eqref{eq:SSL-conc2} with some constants $K_1$ and $K_2$. Notice that \eqref{eq:SSL-conc1} trivially holds with $K_1=2\|f\|_{\bB}$.
Concerning \eqref{eq:SSL-conc2}, we have by a simple application of the conditional Jensen and triangle inequalities
\begin{equs}
\|\E^s &\big( \cA_t(x)-\cA_s(x) -A_{s,t}(x) \big) \|_{L_p(\Omega)}
\\
&\leq \int_s^t\|\cP_{t'-r}\big(f(u_r)-f(O_r+\phi_{s,r}\big)(x)\|_{L_p(\Omega)}\,dr
\\
&\qquad+\int_s^t\|\cP_{t'-r}\big(f(u^n_r)-f(O^n_r+\phi^n_{s,r}\big)(x)\|_{L_p(\Omega)}\,dr.
\end{equs}
Since 
\begin{equs}
u_r-(O_r+\phi_{s,r})&=\int_s^r\cP_{r-\theta}\big(b(u_\theta)-\E^s b(u_\theta)\big)\,d\theta,\\
u_r^n-(O_r^n+\phi^n_{s,r})&=\int_s^r\cP_{\kappa_n(r-\theta)}\big(b(u^n_{\kappa_n(\theta)})-\E^s b(u^n_{\kappa_n(\theta)})\big)\,d\theta,
\end{equs}
it is then clear that \eqref{eq:SSL-conc2} is satisfied with $K_2=4\|f\|_{\cC^1(\R)}\|b\|_{\bB(\R)}$. Consequently, from \eqref{eq:SSL-conc3},  we obtain for all $(s, t) \in [s', t']_<$
\begin{equs}
\| \cA_t(x)-\cA_s(x)\|_{L_p(\Omega)} &  \leq N |t-s|^{3/4} \big( n^{-1/2+ \eps}+\sup_{y \in \T}\| \psi (y)  -\psi^n (y)\|_{L_p(\Omega)}+\| \mathcal{E}^n_{0, \cdot}  \|_{\mathscr{C}^0_p[s', t']} \big)    
\\
& +  N |t-s|^{3/4+\tau}    [ \mathcal{E}^n ]_{\mathscr{C}^\tau_p[s', t']_<}. 
\end{equs}
This, combined with the fact that $\cA_{t'}(x)- \cA_{s'}(x)= \cE^{n,1}_{s',t'}[f] (x)$ and that the constant $N$ is independent of $x \in \mathbb{T}$, leads to \eqref{eq:reg bound} with $s=s'$ and $t=t'$.  Since, $s', t'$ where arbitrary,  the general case of $(s, t) \in [s', t']_<$ also follows from this, since the (semi)-norms on the right hand side of \eqref{eq:reg bound} are non-decreasing functions of the time domain.

  It only remains to remove the additional Lipschitz assumption on $f$. To this end take a smooth approximation of $f$, for instance $f_m:=\cP_{1/m}^\R f$. Then $f_m\to f$ almost everywhere and $\|f_m\|_{\bB(\R)}\leq\|f\|_{\bB(\R)}$. From Girsanov's theorem (see e.g. \cite[Thm~10.14]{DPZ} for a sufficiently general version) we have that for all $x \in \T$,  the law of $u_r(x)$ and that of $\cP_r\psi(x)+O_r(x)$ are mutually absolutely continuous, and therefore for $r>0$, $x\in\T$, the law of $u_r(x)$ is absolutely continuous with respect to the Lebesque measure.
The same holds for $u^n_r(x)$. Therefore, for fixed $s,t,x,n$, $\cE^{n,1}_{s,t}[f_m](x)\to\cE^{n,1}_{s,t}[f](x)$ almost surely, and so an application of Fatou's lemma finishes the proof.  
\end{proof}

\subsection{Estimate for $\cE^{n,2 }$}
The purpose of this section is to provide the estimate for the term $\cE^{n,2}_{s,t}$ in the decomposition \eqref{eq:main-error-decomposition}.
We set 
\begin{equs}
v^n_t(x)= \cP_t^{n}  \psi^n(x) + O^n_t(x).
\end{equs}

\begin{lemma}\label{lem:integral}
Under the assumption of Theorem \ref{thm:main-theorem}, for any  $p>0$ there exists a constant  $N=N(p, \eps,c, K)$  such that  for all $g\in \mathbb{B}(\R)$, and all  $0\leq s < t\leq 1$, $n\in\N$, one has the bound
\begin{equs}
\sup_{x\in \T}\Big\|\int_s^t & p_{t-r}\ast\big( g(v^n_r)\big)(x)-p_{t-r}\ast_n\big(g(v^n_{\kappa_n(r)})\big)(x)\,dr\Big\|_{L_p(\Omega)}
\\
&\leq N\|g\|_{\mathbb{B}(\R)}n^{-1+3\eps}|t-s|^{1/2+\eps/2}.
\label{eq:names}
\end{equs}
\end{lemma}

\begin{proof}
It clearly suffices to prove the claim  for $p\geq 2$ and $\|g\|_{\mathbb{B}(\R)}=1$.  Let us fix $x \in \T$ and  $(s', t') \in [0, 1]_{<}$.  For $(s,t)\in[s',t']_<$ we define  
\begin{equ}
A_{s,t}(x)=\E^s\int_s^t\int_{\T}p_{t'-r}(x-y)\big( g(v^n_r(y))
-g(v^n_{\kappa_n(r)}(\rho_n(y))\big)\,dy\,dr.
\end{equ}
We aim to verify the conditions of the stochastic sewing lemma.
By the tower property of conditional expectation, it is easy to see that
\begin{equ}
\E^s\delta A_{s,u,t}(x)=0.
\end{equ}
This shows that \eqref{eq:SSL-cond2} is satisfied with $C_2=0$. 

Moving on to \eqref{eq:SSL-cond1}, we separate two cases. When $t\geq s+2h$, we write
\begin{equs}
A_{s,t}(x)&=I_1+I_2
:=\Big(\int_s^{s+2h}+\int_{s+2h}^t\Big)
\int_{\T}p_{t-r}(x-y)\E^s\big( g(v^n_r(y))
-g(v^n_{\kappa_n(r)}(\rho_n(y))\big) \,dy\,dr.
\end{equs}
For $r \in [s+2h, t]$ we have $\kappa_n(r)\geq s$. Therefore, we have by \eqref{eq:OUn-transition-main}
\begin{equs}  
J_{r,s}(y)&:=\E^s\big( g(v^n_r(y))
-g(v^n_{\kappa_n(r)}(\rho_n(y))\big)
\\
&=
\cP_{Q^n(r-s)}^\R
g\big( \cP_{r}^{n} \psi^n(y)+ \widehat{O}^n_{s,r}(y)  \big)
 -\cP_{Q^n(\kappa _n(r)-s)}^\R
g\big(\cP_{\kappa_n(r)}^{n} \psi^n (\rho_n(y)) + \widehat{O}^n_{s,\kappa_n(r)}(\rho_n(y))\big).
\end{equs}
Applying \eqref{eq:HK bound mixed} for the outer heat kernels with $\alpha=0$ and $\beta=1$  we have
\begin{equs}
|J_{r,s}(y)| & \lesssim \Big( |\cP_{r }^{n} \psi^n (y)- 
\cP_{\kappa_n(r)}^{n} \psi^n (\rho_n(y))|
+ \big|\widehat{O}^n_{s,r}(y)-\widehat{O}^n_{s,\kappa_n(r)}(\rho_n(y))\big|
\\
& \qquad +|Q^n(r-s)-Q^n(\kappa_n(r)-s)|^{1/2}
\Big)|Q^n(\kappa_n(r)-s)|^{-1/2}.     \label{eq:I-first}
\end{equs}
For the first term we apply Lemma \ref{lem:terms-from-IC} (with $\beta=1$, $\alpha=1/2-\eps$) to get that 
\begin{equs}        
|\cP_{r }^{n} \psi^n(y)- 
\cP_{\kappa_n(r)}^{n} \psi^n(\rho_n(y))| & \lesssim  n^{-1+\eps} r^{-1/4-\eps/2}. \label{eq:bound-for-IC-terms}
\end{equs} 
For the second term on the right hand side of \eqref{eq:I-first}we first write
\begin{equs}
\widehat{O}^n_{s,r}(y)-\widehat{O}^n_{s,\kappa_n(r)}(\rho_n(y))&=\big(\widehat{O}^n_{s,r}(y)-\cP_{r-s}O_s(y)\big)+\big(\cP_{\kappa_n(r)-s}O_s(\rho_n(y))-\widehat{O}^n_{s,\kappa_n(r)}(\rho_n(y))\big)\\
&\qquad + \big(\cP_{r-s}O_s(y)-\cP_{\kappa_n(r)-s}O_s(\rho_n(y))\big).
\end{equs}
The $L_p(\Omega)$ norms of the first two terms are readily bounded by \eqref{eq:Estimate-BDG} (with $\beta=2$).
As for the third,
 by \eqref{eq:HK bound mixed} (with $\beta=1$ and $\alpha=1/2-\eps$) and Proposition \ref{prop:OU Holder} we obtain 
\begin{equs}
\|\cP_{r-s} &  O_s(y)-\cP_{\kappa_n(r)-s} O_s(\rho_n(y))\|_{L_p(\Omega)} 
\\
&\lesssim    \|O_s\|_{L_p(\Omega; C^{1/2-\eps}) 
} \big(|y-\rho_n(y)| +|r-\kappa_n(r)|^{1/2}\big)|\kappa_n(r)-s|^{-1/4-\eps/2}
\\
&\lesssim    n^{-1 }|r-s|^{-1/4-\eps/2},
\end{equs}
where in the last line we used that $|r-s|\lesssim |\kappa_n(r)-s|$ for $r\in[s+2h,t]$.
Consequently, we have
\begin{equs}       
\| \widehat{O}^n_{s,r}(y)-\widehat{O}^n_{s,\kappa_n(r)}(\rho_n(y))\|_{L_p(\Omega)} \lesssim  n^{-1 }|r-s|^{-1/4-\eps/2}.
  \label{eq:P^n_r-P^n-kappa}
\end{equs}
Next, by \eqref{eq:Q-upper-bound} (with $\beta=1$) we have 
\begin{equs}     \label{eq:IDK}
|Q_n(r-s)-Q_n(\kappa_n(r)-s)|^{1/2} \lesssim  n ^{-1}|r-s|^{-1/4}.
\end{equs}
Finally, by \eqref{eq:Q-lower-bound} we have
\begin{equ}\label{eq:IDK2}
|Q_n(\kappa_n(r)-s)|^{-1/2}\lesssim |r-s|^{-1/4}.
\end{equ}
Substituting the bounds \eqref{eq:bound-for-IC-terms}-\eqref{eq:P^n_r-P^n-kappa}-\eqref{eq:IDK}-\eqref{eq:IDK2} into \eqref{eq:I-first},
we get 
\begin{equs}
\|J_{r,s}(y)\|_{L_p(\Omega)}  \lesssim n^{-1+\eps} |r-s|^{-1/2-\eps/2}\lesssim n^{-1+3\eps}|r-s|^{-1/2+\eps/2}.
\end{equs}
It remains to integrate with respect to $y$ and $r$ to get
\begin{equs}
\| I_2\|_{L_p(\Omega)} \lesssim  n^{-1+3\eps }  |t-s|^{1/2+\eps/2}.
\end{equs}
The term $I_1$ is trivial: by using  the  boundedness of $g$ we get
\begin{equs}
\| I_1\|_{L_p(\Omega)} \lesssim h \lesssim  n^{-1+3\eps} |t-s|^{1/2+3\eps/2} \lesssim  n^{-1+3\eps} |t-s|^{1/2+\eps/2}.
\end{equs}
Consequently, for $ t\geq s+2h$ we have shown that 
\begin{equs}\label{eq:quadribadri}
\|A_{s, t}(x) \|_{L_p(\Omega)} \lesssim  n^{-1+3\eps} |t-s|^{1/2+\eps/2}.
\end{equs}
In addition, note that \eqref{eq:quadribadri} also holds in the $t\in[s,s+2h]$ case, since it follows from the trivial bound
$\|A_{s, t}(x) \|_{L_p(\Omega)} \lesssim |t-s|$.
We can conclude that \eqref{eq:SSL-cond1} is satisfied with $C_1=Nn^{-1+3\eps}$ and $\eps_1=\eps/2$. Therefore, Theorem \ref{thm:SSL} applies.

Let us now define$\cA(x) : [s', t'] \to  L_p(\Omega)$  by 
\begin{equs}
\cA_t(x):= \int_{s'}^t\int_{\T}p_{t'-r}(x-y)\big( g(O^n_r(y))
-g(O^n_{\kappa_n(r)}(\rho_n(y))\big)\,dy\,dr, \qquad t \in [s', t']. 
\end{equs}
It is obvious that $\cA_{s'}(x) =0$, and by the adaptedness of $O^n$ it follows that $\cA_t(x)$ is $\cF_t$-measurable  for all $t \in [s',t']$. Moreover, since $\|g\|_{\mathbb{B}(\R)}=1$, it is obvious that it
satisfies \eqref{eq:SSL-conc1} with $K_1=4$. In addition, by definition,  it satisfies \eqref{eq:SSL-conc2} with $K_2=0$. By the conclusion of Theorem \ref{thm:SSL},  it follows that it is  the unique process with these properties.  Moreover, by  \eqref{eq:SSL-conc3} we have that 
\begin{equs}
\, & \|  \int_{s'}^{t'}\int_{\T}p_{t'-r}(x-y)\big( g(O^n_r(y))
-g(O^n_{\kappa_n(r)}(\rho_n(y))\big)\,dy\,dr\|_{L_p(\Omega)}
\\
&\qquad =  \|\cA_{t'}(x)-\cA_{s'}(x)\|  \leq   N n^{-1+3\eps}|t'-s'|^{1/2+\eps/2}
\end{equs}
Notice that $s', t'$ where arbitrary,  as was $x \in \T$, hence, the claim follows. 
\end{proof}

\begin{corollary}      \label{col:quadrature}
Under the assumption of Theorem \ref{thm:main-theorem}, for any  $p>0$ there exists a constant  $N=N(p, \eps,c,\|b\|_{\bB(\R)}, K)$   such that for  all  $0\leq s\leq t\leq 1$, $n\in\N$, one has the bound
\begin{equs}
\sup_{x\in\T}\|\cE^{n,2}_{s,t}(x)\|_{L_p(\Omega)}
\leq N n^{-1+3\eps}|t-s|^{1/2+ \eps/2}.
\label{eq:quadrature}
\end{equs}
\end{corollary}
\begin{proof}
Fix $p\in(0,\infty)$, $x\in\T$, $0\leq s\leq t\leq 1$, $n\in\N$.
For any random field $\big(Y_t(x)\big)_{(t,x)\in[0,1]\times \T}$ denote the random variable
\begin{equ}
h(Y)=\int_s^t  p_{t-r}\ast b(Y_r)(x)-p_{t-r}\ast_nb(Y_{\kappa_n(r)})(x)\,dr.
\end{equ}
From Girsanov's theorem, as in Lemma \ref{lem:uniform-C-1/2}, 
we have
\begin{equs}
\E|h(u^n)|^p=\tilde\E\big(|h(u^n)|^p\rho^{-1}\big)&\leq \big(\tilde \E|h(u^n)|^{2p}\big)^{1/2}\big(\tilde\E\rho^{-2}\big)^{1/2}
\\
&=\big(\E|h(v^n)|^{2p}\big)^{1/2}\big(\E\rho^{-1}\big)^{1/2}
\\&\lesssim \big(n^{-1+3\eps}|t-s|^{1/2+\eps/2}\big)^p,
\end{equs}
where $\tilde \E$ denotes the expectation under the measure $\tilde\bP$ defined in \eqref{eq:rad-nyk} and $\rho := d\bP/ d \tilde \bP$.  
 To get the last line we used Lemma \ref{lem:integral} with $2p$ in place of $p$ and $b$ in place of $g$.
\end{proof}

\subsection{Proof of Theorem \ref{thm:main-theorem}}
As indicated in Section \ref{sec:outline}, we first aim to derive a buckling inequality for $\cE^n$.
In the decomposition \eqref{eq:main-error-decomposition} the only term not treated so far is $\cE^{n,3}$, for which however it is easy to see the almost sure bound
\begin{equ}\label{eq:E3}
\sup_{x\in\T}\cE^{n,3}_{s,t}(x)\lesssim n^{-1/2}|t-s|^{1/2}.
\end{equ}
Indeed, when $|t-s|\leq h$, then simply using the boundedness of $b$ yields a bound of order $|t-s|$, which even implies a bound of order $n^{-1}|t-s|^{1/2}$. In the regime $|t-s|\geq h$ we split the integral into two as usual, and the trivial estimates
\begin{equs}
&\Big|\int_{t-h}^t \Big( p_{t-r} *_nb (u^n_{\kappa_n(r)}) -  p^n_{\kappa_n(t-r)} *_nb (u^n_{\kappa_n(r)}) \Big)  \, dr\Big|\lesssim n^{-2},
\\
&\Big|\int_{s}^{t-h} \Big( p_{t-r} *_nb (u^n_{\kappa_n(r)}) -  p^n_{\kappa_n(t-r)} *_nb (u^n_{\kappa_n(r)}) \Big)  \, dr\Big|\lesssim\int_{s}^{t-h}\|p_{t-r}-p^n_{\kappa_n(t-r)}\|_{L_1(\T)}\,dr
\end{equs}
indeed imply \eqref{eq:E3}, using \eqref{eq:Pn-P} (with $\beta=1$) to bound the last integral.

Fix $\tau=3/8$.
Denote briefly by $\tilde K$ the constant $N$ obtained from Lemma \ref{lem:uniform-C-1/2}. When we apply below Lemma \ref{lem:regularisation-lemma} and Corollary \ref{col:quadrature}, we do so with $\tilde K$ in place of $K$.
For a parameter $N_0\in(0,1]$ to be specified later
take $S\in\Lambda_n\cap[N_0/2,N_0]$, which is certainly possible for large enough $n$.
Putting together \eqref{eq:reg bound}, \eqref{eq:quadrature}, \eqref{eq:E3}, and the trivial inequality $\|\cE^n_{0,\cdot}\|_{\scC^0_p[0,S]}\leq[\cE^n]_{\scC^{3/8}_p[0,S]_<}$, we have for all $(s,t)\in[0,S]_<$
\begin{equ}\label{eq:y01}
\sup_{x\in\T}\|\cE^n_{s,t}\|_{L_p(\Omega)}\leq N_1\big(n^{-1/2+\eps}+\sup_{x \in \T}\| \psi (x)  -\psi^n (x)\|_{L_p(\Omega)}+[\cE^n]_{\scC^{3/8}_p[0,S]_<}\big)|t-s|^{1/2},
\end{equ}
where the constant $N_1$ does not depend on $n$ or $N_0$ (in our usual notation, $N_1\lesssim 1$).
Dividing by $|t-s|^{3/8}$ and taking supremum over $(s,t)\in[0,S]_<$, we get
\begin{equ}
\,[\cE^n]_{\scC^{3/8}_p[0,S]_<}\leq N_1\big(n^{-1/2+\eps}+
\sup_{x \in \T}\| \psi (x)  -\psi^n (x)\|_{L_p(\Omega)}+
[\cE^n]_{\scC^{3/8}_p[0,S]_<}\big)S^{1/8}.
\end{equ}
We now fix $N_0=(2N_1)^{-8}$. Since $S\leq N_0$, the inequality buckles and we get
\begin{equ}\label{eq:final-drift-bound}
\,[\cE^n]_{\scC^{3/8}_p[0,S]_<}\leq N_1 \big(n^{-1/2+\eps}+\sup_{x \in \T}\| \psi (x)  -\psi^n (x)\|_{L_p(\Omega)}\big).
\end{equ}
Returning to the main error, we have
\begin{equ}
u_t(x)-u^n_t(x)= \big(\cP_t\psi(x)-\cP_t\psi^n(x)\big)+\big(\cP_t\psi^n(x)-\cP_t^n\psi^n(x)\big) + \cE^n_{0,t}(x)+\big(O_t(x)-O^n_t(x)\big).
\end{equ}
These terms are bounded by trivially, \eqref{eq:rate-deterministic} (with $\alpha=1/2-\eps$), \eqref{eq:final-drift-bound}, and \eqref{eq:Estimate-BDG2} (with $\beta=1-2\eps$), respectively.
This gives the bound
\begin{equ}
\sup_{t,x\in[0,S]\times\T}\|u_t(x)-u^n_t(x)\|_{L_p(\Omega)}\leq N_2 \big(n^{-1/2+\eps}+\sup_{x \in \T}\| \psi (x)  -\psi^n (x)\|_{L_p(\Omega)}\big)
\end{equ}
with another constant $N_2\lesssim 1$.
Repeating the same argument on $[S,2S]$, with viewing $u_S$ and $u^n_S$ as initial conditions, we get
\begin{equs}
\sup_{t,x\in[S,2S]\times\T}\|u_t(x)-u^n_t(x)\|_{L_p(\Omega)}&\leq N_2 \big(n^{-1/2+\eps}+\sup_{x \in \T}\| u_S (x)  - u^n_S (x)\|_{L_p(\Omega)}\big)
\\
&\leq (N_2^2+N_2) \big( n^{-1/2+\eps}+\sup_{x \in \T}\| \psi (x)  -\psi^n (x)\|_{L_p(\Omega)}\big).
\end{equs}
By iterating the argument at most $2/N_0$ times and recalling that $N_0$ does not depend on $n$, the proof is finished.\qed

\bibliographystyle{Martin}
\bibliography{SPDE-bib} 

\end{document}